\documentclass[11pt]{amsart}
\usepackage{amsmath,amssymb,txfonts}
      \newtheorem{theorem}{Theorem}[section]
      \newtheorem{remark}[theorem]{Remark}

      \newcommand{\ct}[1]{\langle {#1}\rangle \lower.3ex\hbox{$_{t}$}}
      \newcommand{\lt}[1]{[ {#1}] \lower.3ex\hbox{$_{t}$}}

\begin{document}

\title[Towards Conformal Capacities in Euclidean Spaces]{Towards Conformal Capacities in Euclidean Spaces}

\author{Jie Xiao}
\address{Department of Mathematics and Statistics, Memorial University of Newfoundland, St. John's, NL A1C 5S7, Canada}
\email{jxiao@mun.ca}
\thanks{This project was in part supported by MUN's University Research Professorship (208227463102000) and NSERC of Canada.}


\date{}



\maketitle

\tableofcontents

\section{Introduction}\label{s1}
\setcounter{equation}{0}

   \subsection{A background}\label{s10} Thanks to its role in two-dimensional potential theory that is the study of planar harmonic functions in mathematics and mathematical physics, the conformal or logarithmic capacity in the Euclidean plane $\mathbb R^2$ has been studied systemically; see \cite{Je96b}, \cite{AikE}, \cite{Hil}, \cite{ Lan}, \cite{Ran} and \cite{SafT} for some relatively recent publications on this topic. However, the higher dimensional extension (i.e., to the Euclidean space $\mathbb R^n$, $n\ge 3$) of the planar conformal capacity has received relatively little attention due to a nonlinear nature; see \cite{Bet}, \cite{ColC}, \cite{Fug1, Fug2}, \cite{AndV} and \cite{Kol} (see also \cite{AdH} and \cite{HKM} for some function-space-based capacities) only because of the author's limited knowledge of other references.

\subsection{The first definition}\label{s11} In his 2010 paper \cite{Bet}, Betsakos introduced the concept of the reduced conformal modulus of a compact subset of $\mathbb R^n$. To be more precise, let us recall some notations. Given $n\ge 2$. A pair $(O,K)$ of sets in $\mathbb R^n$ is said to be a condenser if $O$ is open and $K$ is a nonempty compact subset of $O$. The well-known conformal capacity of a given condenser $(O,K)$ is defined as
\begin{equation}\label{e11}
\hbox{ncap}(O,K)=\inf_{u\in W(O,K)}\int_{\mathbb R^n} |\nabla u|^n\,dV
\end{equation}
where $dV$ and $W(O,K)$ stand respectively for the volume element
and all functions $u: \mathbb R^n\to\mathbb R^1$ that are not only
continuous and absolutely continuous on almost all lines in each
cube in $\mathbb R^n$ parallel to the coordinates axes, but also
enjoy $0\le u\le 1$, $u|_K=0$, and the closure of $\{x\in\mathbb
R^n: 0\le u(x)<1\}$ is a compact subset of $A$. According to Gehring
\cite{Ge1}, we see that if the infimum in (\ref{e11}) is finite then
there is a unique extremal function (or $n$-capacity potential) $u$ enjoying the conformally invariant Euler-Lagrange equation in a weak sense:

\begin{equation}\label{e12}
-\hbox{div}(|\nabla u|^{n-2}\nabla u)=0\quad\mbox{in}\quad O\setminus{K}
\end{equation}
subject to
$$
\begin{cases}
u=1\quad\mbox{on}\quad {\partial K};\\
u=0\quad\mbox{on}\quad {\partial O},
\end{cases}
$$
and hence an integration-by-part gives (cf. \cite{PhiP})
\begin{equation}\label{e13}
\hbox{ncap}(O,K)=\int_{\{x\in\overline{O\setminus K}:\ u(x)=t\}}|\nabla
u|^{n-1}\,dS\quad\forall\quad t\in [0,1],
\end{equation}
where $dS$ represents the area element, i.e., the $n-1$ dimensional Hausdorff measure. Now, the conformal modulus of the condenser $(O,K)$ is determined by
\begin{equation}\label{e14}
\hbox{nmod}(O,K)=\left(\frac{\hbox{ncap}(O,K)}{\sigma_{n-1}}\right)^\frac1{1-n}
\end{equation}
where
$$
\sigma_{n-1}=n\omega_n=\frac{2\pi^\frac{n}{2}}{\Gamma(\frac{n}{2})}
$$
for which $\Gamma(\cdot)$ is the usual Gamma function. The subadditivity of the conformal modulus (cf. \cite[Lemma 2.1]{Flu}, \cite[pages 159-161]{AndVV}, \cite[Appendix]{MarRSY}) yields that if $K\subseteq r\mathbb B^n$ then
$$
r\mapsto \hbox{nmod}(r\mathbb B^n, K)-\ln r
$$
is an increasing function, where $r\mathbb B^n$ is the open ball centered at the origin with radius $r$, i.e., the $r$-expansion of the unit open ball $\mathbb B^n$. Consequently, a normalized version of the Betsakos reduced conformal modulus of a given compact set $K\subseteq\mathbb R^n$ is determined via
\begin{equation}\label{e15}
\hbox{ncap}_1(K)=\exp\Big(-\lim_{r\to\infty}\big(\hbox{nmod}(r\mathbb
B^n, K)-\ln r\big)\Big).
\end{equation}
Since the case $n=2$ of (\ref{e15}) is just the conformal capacity
of $K\subseteq\mathbb R^2$, we may regard $\hbox{ncap}_1(K)$ as the
conformal capacity of $K\subseteq\mathbb R^n$ in dimension $n$.

\subsection{The second definition}\label{s12} In
their 2005 work \cite{ColC}, Colosanti-Cuoghi used the equilibrium
potential to introduce another conformal capacity. To see this,
from now on, for the closure $K$ of a bounded open subset of $\mathbb R^n$ let $u=u_K$ be its $n$-equilibrium potential, i.e., the
unique weak solution to the following boundary value problem:

\begin{equation}\label{e16}
\left\{
\begin{array}{lll}
-\hbox{div}(|\nabla u|^{n-2}\nabla u) =0\quad\mbox{in}\quad \mathbb R^n\setminus{K};\\
\\
u=0\quad\mbox{on}\quad \partial K\quad\& \quad u(x)\sim\ln
|x|\quad\hbox{as}\quad |x|\to\infty,
\end{array}
\right.
\end{equation}
where $\sim$ means that there exists a constant $c>0$ such that
$$
c^{-1}\le\frac{u(x)}{\ln|x|}\le c\quad\hbox{as}\quad |x|\to\infty.
$$
In accordance with Kichenassamy-Veron's \cite[Theorem 1.1 and Remarks 1.4-1.5]{KicV},
$u(x)-\ln|x|$ tends to a constant depending on $K$ as
$|x|\to\infty$, and so the following
\begin{equation}\label{e17}
\hbox{ncap}_2(K)=\exp\Big(-\lim_{|x|\to\infty}\big(u(x)-\ln|x|\big)\Big)
\end{equation}
is employed to define the second conformal capacity of $K$ since
the case $n=2$ of (\ref{e17}) is just the conformal capacity on
$\mathbb R^2$.

\subsection{The third definition}\label{s13} For a compact
subset $K$ of $\mathbb R^n$, let
\begin{equation}\label{e18}
\hbox{nrob}(K)=\inf_\mu\int_K\int_K\Big(\ln\frac1{|x-y|}\Big)\,d\mu(x)d\mu(y)
\end{equation}
be the conformal or logarithmic Robin mass of $K$, where the infimum ranges
over all unit nonnegative Borel measures $\mu$ in $\mathbb R^n$ with
its support in $K$, and actually, this infimum is attainable.
According to Anderson-Vamanamurthy's 1988 article \cite{AndV} and
Fuglede's 1960 papers \cite{Fug1, Fug2}, the potential-theoretic
conformal capacity of $K$ is:
\begin{equation}\label{e19}
\hbox{ncap}_3(K)=\exp\big(-\hbox{nrob}(K)\big).
\end{equation}
Of course, when $n=2$, (\ref{e19}) coincides (\ref{e17}) and
(\ref{e15}).

\subsection{An overview}\label{s14}
In this paper we will analyze seven problems which are associated with the above-introduced conformal capacities. First of all, we prove that (\ref{e15}) is the same as (\ref{e17}) for $K$ being the closure of a bounded open subset of $\mathbb R^n$ (cf. Theorem \ref{t21}), but not the same as (\ref{e18}) unless $n=2$ (cf. Theorem \ref{t22}). Secondly, we use (\ref{e15}) to split the iso-diameter and iso-mean-width inequalities (cf. Theorems \ref{t31} \& \ref{t32} \& \ref{t33}). Thirdly, we handle conformal capacity from above through an integral of the mean curvature (cf. Theorem \ref{t4a1}) and the graphic ADM mass (cf. Theorem \ref{t4a2}). Fourthly, we give an integral identity and a lower bound estimate for the non-tangential limit of the gradient of an $n$-equilibrium potential on the boundary of a convex body (cf. Theorems \ref{p-41} \& \ref{l5c}). Fifthly, we establish Hadamard's variational formula for (\ref{e17}) (cf. Theorem \ref{t41g}). Sixthly, we deal with the existence and uniqueness for the Minkowski type problem arising from (\ref{e15})/(\ref{e17}) (cf. Theorem \ref{t53} extending Jerison's \cite[Corollary 6.6]{Je96b}). Last of all, we discuss Yau's prescribed mean curvature problem \cite[Problem 59]{Yau} in a weak sense and then get its conformal capacity analogue (cf. Theorem \ref{t71}). Here it is perhaps appropriate to point out that since our extension is from linear case $n=2$ (where the classical $2=n$-harmonic functions are often taken into account) to nonlinear case $n\ge 3$ (where only the nonlinear $3\le n$-harmonic functions can be used), in many situations we have to seek a way, which turns out to be highly non-trivial, to settle these issues.

\medskip

\noindent{\bf Acknowledgments}. The author is grateful to David Jerison for his very helpful comments on this work.

\section{Comparisons among three conformal capacities}\label{s2}
\setcounter{equation}{0}

\subsection{Comparing the first and second conformal capacities}\label{s21} We
always have the following identification.

\begin{theorem}\label{t21} Let $K$ be the closure of a bounded open subset of $\mathbb
R^n$. Then $\hbox{ncap}_1(K)=\hbox{ncap}_2(K)$.
\end{theorem}
\begin{proof} Given an $r\in (0,\infty)$ large enough for $K\subseteq r\mathbb B^n$, let $u_r$
be the unique solution to

\begin{equation}\label{e21}
\left\{
\begin{array}{lll}
-\hbox{div}(|\nabla u_r|^{n-2}\nabla u_r) =0\quad\mbox{in}\quad r\mathbb B^n\setminus{K};\\
\\
u_r=0\quad\mbox{on}\quad \partial K\quad\& \quad u_r(x)=\ln
r\quad\hbox{as}\quad |x|=r,
\end{array}
\right.
\end{equation}
According to the argument for \cite[Theorem 2.2]{ColC}, $\{u_r\}$
has a subsequence, still denoted by $\{u_r\}$, convergent to $u$
which is the unique weak solution of (\ref{e16}) and makes
$\alpha=\lim_{|x|\to\infty}\big(u(x)-\ln |x|\big)$ be finite. According to \cite{KicV}, we have that if $|x|\to\infty$ then
\begin{equation}\label{e22}
u(x)=\ln|x|+\alpha +o(1)\quad \&\quad |\nabla u(x)|=|x|^{-1}+o(|x|^{-1}).
\end{equation}
Consequently, by the maximum principle one has
$$
0<u(x)\le\max_{\partial r\mathbb B^n}u\quad\forall\quad x\in r\mathbb B^n.
$$
If
$$
v_{r}(x)=\frac{u(x)}{\max_{\partial r\mathbb B^n}u}
$$
then for the sufficiently large $r$,
\begin{equation}\label{e23}
\left\{
\begin{array}{lll}
-\hbox{div}(|\nabla v_{r}|^{n-2}\nabla v_{r})=0\quad\mbox{in}\quad r\mathbb B^n\setminus{K};\\
\\
v_{r}=0\quad\mbox{on}\quad \partial K\quad\& \quad 0\le
v_{r}(x)\le 1\quad\hbox{as}\quad |x|=r,
\end{array}
\right.
\end{equation}
and hence the uniqueness of $u$ plus (\ref{e22}) implies that for
$$
r=\exp\Big((\frac{1+t}{1-t})o(1)-\alpha\Big)\quad\hbox{as}\quad 0<t\to 1,
$$
one has
\begin{eqnarray*}
\hbox{ncap}(r\mathbb B^n,K)&=&\int_{\{x\in\mathbb R^n: v_{r}(x)=t\}}|\nabla
v_{r}|^{n-1}\,dS\\
&=&\int_{\partial r\mathbb B^n}\left(\frac{1+o(1)}{r(\ln r+\alpha+o(1))}\right)^{n-1}\,dS\\
&=&\sigma_{n-1}\left(\frac{1+o(1)}{\ln
r+\alpha+o(1)}\right)^{n-1}.
\end{eqnarray*}
An application of (\ref{e14}) yields
$$
\hbox{nmod}(r\mathbb B^n,K)=\frac{\ln r +\alpha
+{o}(1)}{1+{o}(1)},
$$
thereby deriving through using (\ref{e15}) plus letting $r\to\infty$,
$$
\hbox{ncap}_1(K)=\exp(-\alpha)=\hbox{ncap}_2(K).
$$
\end{proof}

\subsection{Comparing the first and third conformal capacities}\label{s22} For a condenser $(O,K)$ in $\mathbb R^n$ we write
$$
\hbox{nmd}(O,K)=\inf_{\nu\in\mathcal
F(O,K)}\int_{\mathbb R^n}\int_{\mathbb R^n}\Big(\ln\frac1{|x-y|}\Big)\,d\nu(x)d\nu(y)
$$
for the transfinite $n$-modulus of $(O,K)$, where $\mathcal F(O,K)$
is the family of all signed Borel measures $\nu=\nu_K-\nu_O$ with
$\nu_O$ and $\nu_K$ being unit nonnegative Borel measures on $\mathbb
R^n$. The above infimum is attainable provided it is finite; see
also \cite[Lemma 2]{AndV}.

\begin{theorem}\label{t22} Let $\overline{\mathbb B^n}$ be the closed unit ball of $\mathbb R^n$.

\item{\rm(i)} If $K$ is a compact subset of $\mathbb R^n$ then
$$
\hbox{ncap}_3(K)=\frac{\exp\Big(-\lim_{r\to\infty}\big(\hbox{nmd}(r\mathbb
B^n,K)-\ln r\Big)}{\hbox{ncap}_3(\overline{\mathbb
B^n})}.
$$

\item{\rm(ii)} If $r>0$ and $\sigma_0=0$ then
$$
\hbox{ncap}_1(r\overline{\mathbb
B^n})=r\ \&\ \hbox{ncap}_3(r\overline{\mathbb
B^n})=r\exp\left(-\frac{\sigma_{n-2}}{2\sigma_{n-1}}\int_0^\pi\frac{\ln\big(2(1-\cos\theta)\big)}{(\sin\theta)^{2-n}}\,d\theta\right)
$$
and hence both are not equal unless $n=2$.
\end{theorem}
\begin{proof} (i) This follows from
$$
\hbox{nrob}(r\overline{\mathbb B^n})=-\ln r +\hbox{nrob}(\overline{\mathbb B^n}).
$$
and the second chain of inequalities in \cite[Theorem 5]{AndV}, in
particular:
$$
2\ln
m(r)\le\hbox{nmd}(r\overline{\mathbb B^n},K)-\hbox{nrob}(r\overline{\mathbb B^n})-\hbox{nrob}(K)\le 2\ln M(r)
$$
for which as $r\to\infty$ one has
$$
\begin{cases}
& m(r)=\inf_{x\in K, |y|=r}|x-y|=\ln r+{o}(1);\\
& M(r)=\sup_{x\in K, |y|=r}|x-y|=\ln r+{o}(1).
\end{cases}
$$

(ii) Note that for $r>1$ one has
\begin{equation}\label{e24}
\hbox{nmod}(r\mathbb B^n,\overline{\mathbb B^n})=\ln r
\end{equation}
and a slight modification of the argument for \cite[Theorem 6]{AndV}
gives
\begin{equation}\label{e25}
\hbox{nmd}(r\mathbb B^n,\overline{\mathbb B^n})=\ln
r+\frac{\sigma_{n-2}}{\sigma_{n-1}}\int_0^\pi(\sin
\theta)^{n-2}\ln\frac{1+r^2-2r\cos\theta}{2r^2(1-\cos\theta)}\,d\theta.
\end{equation}
So, (\ref{e24}) and (\ref{e25}) imply
\begin{eqnarray}\label{e26}
0&\le&\hbox{nmod}(r\mathbb B^n,\overline{\mathbb
B^n})-\hbox{nmd}(r\mathbb B^n,\overline{\mathbb B^n})\nonumber\\
&\to& \frac{\sigma_{n-2}}{\sigma_{n-1}}\int_0^\pi(\sin
\theta)^{n-2}\ln\big(2(1-\cos\theta)\big)\,d\theta\\
&\hbox{as}& r\to\infty\nonumber.
\end{eqnarray}
Now, (\ref{e26}), the formula in (i), and the first and third
definitions of the conformal capacity are used to deduce the
desired estimate in (ii).
\end{proof}

\begin{remark}\label{r13} Thanks to Theorems \ref{t21}-\ref{t22}, in the forthcoming sections the notation $\hbox{ncap}(\cdot)$ stands for either $\hbox{ncap}_1(\cdot)$ or $\hbox{ncap}_2(\cdot)$, but not $\hbox{ncap}_3(\cdot)$ which will not be used again. Of course, for an open set $O\subseteq\mathbb R^n$ we may define
$$
\hbox{ncap}(O)=\sup\{\hbox{ncap}(K):\ \hbox{compact}\ K\subseteq O\}
$$
and
$$
\hbox{ncap}(E)=\inf\{\hbox{ncap}(O):\ \hbox{open}\ O\supseteq E\}
$$
for an arbitrary set $E\subset\mathbb R^n$. Now, it is easy to see that $\hbox{ncap}(\cdot)$ is a monotone and continuous set function. However, there is an example showing that $\hbox{2cap}(\cdot)$ is not subadditive (cf. \cite[page 60]{AikE} or \cite{Pyr}).
\end{remark}

\section{Iso-diameter and iso-mean-width inequalities via conformal capacities}\label{s3}
\setcounter{equation}{0}

\subsection{Iso-diameter inequality}\label{s31}
The iso-diameter or Bieberbach's inequality (cf. \cite[page 69]{EvaG} and \cite[page 318]{Schn}) says that if $K$ is a compact subset of $\mathbb R^n$ then its diameter $\hbox{diam}(K)$ and volume $V(K)$ enjoy the following inequality
\begin{equation}\label{e31}
\left(\frac{V(K)}{\omega_n}\right)^\frac1n\le \frac{\hbox{diam}(K)}{2}
\end{equation}
with equality if $K$ is a ball. Meanwhile, the isoperimetric inequality states that if $K$ is a compact sub-domain of $\mathbb R^n$ then its surface area $S(K)$ and volume $V(K)$ satisfy the isoperimetric inequality
\begin{equation}\label{e32}
\left(\frac{V(K)}{\omega_n}\right)^\frac1n\le \left(\frac{S(K)}{\sigma_{n-1}}\right)^\frac1{n-1}
\end{equation}
with equality when and only when $K$ is a ball. In their 2011 manuscript \cite{MagPP}, Maggi-Ponsiglione-Pratelli reproved that if $K$ is convex then
\begin{equation}\label{e33}
\left(\frac{V(K)}{\omega_n}\right)^\frac1n\le \left(\frac{S(K)}{\sigma_{n-1}}\right)^\frac1{n-1}\le \frac{\hbox{diam}(K)}{2}
\end{equation}
with two equalities if $K$ is a ball. Historically, the right inequality of (\ref{e33}) is called Kubota's inequality; see also \cite{Kub}.

In the sequel, we show that (\ref{e31}) can be split by using the conformal capacity, i.e.,
\begin{equation}\label{e34}
\left(\frac{V(K)}{\omega_n}\right)^\frac1n\le \hbox{ncap}(K)\le \frac{\hbox{diam}(K)}{2}
\end{equation}
with two equalities if $K$ is a ball.

\subsection{Volume to conformal capacity}\label{s32} Simply motivated by \cite{Pol}, \cite{EncP}, \cite{Kaw} and \cite{GarS}, we obtain the following assertion whose case $n=2$ is known (cf. \cite[Theorem 5.3.5]{Ran}).

\begin{theorem}\label{t31} Let $K$ be a compact subset of $\mathbb R^n$. Then
\begin{equation}\label{e35}
\left(\frac{V(K)}{\omega_n}\right)^\frac1n\le\hbox{ncap}(K)
\end{equation}
with equality if $K$ is a ball.
\end{theorem}

\begin{proof} Clearly, equality of (\ref{e35}) occurs when $K$ is a ball. So, it remains to verify (\ref{e35}). Given an $r>0$ large enough for $K\subseteq r\mathbb B^n$. Suppose
$\hbox{Sch}(E)$ is the Schwarz symmetrization of $E\subseteq\mathbb
R^n$, i.e., the origin-centered ball with radius
$\big(V(E)/\omega_n\big)^\frac1n$. According to the iso-capacitary
inequality in \cite[Theorem 3.6]{Wan} which was originally
established in \cite{Ge0} (for $n=3$) and \cite{Mos} (for $n\ge
3$), we have
$$
\hbox{ncap}(r\mathbb B^n,K)\ge\hbox{ncap}\big(r\mathbb
B^n,\hbox{Sch}(K)\big)=\sigma_{n-1}\left(\ln
\frac{r}{\big(V(K)/\omega_n\big)^\frac1n}\right)^{1-n}
$$
thereby getting via (\ref{e14})
$$
\hbox{nmod}(r\mathbb
B^n,K)\le\ln\frac{r}{\big(V(K)/\omega_n\big)^\frac1n}.
$$
This, along with (\ref{e15}), yields
$$
\hbox{ncap}(K)\ge\Big(V(K)/\omega_n\Big)^\frac1n,
$$
as desired. The last inequality can be also proved by
\cite[Lemma 1]{Bet}.
\end{proof}

\begin{remark}\label{r31} Let $K$ be $C^2$ convex and $\nabla u_K$ still stand for its non-tangential limit at $\partial K$ (cf. \cite[Theorem 3]{LewN} and \cite[Theorem 4.3]{Lew}). If $|\nabla u_K|$ equals a positive constant $c$ on $\partial K$, then $c^{-1}=\big({V(K)}/{\omega_n}\big)^\frac1n$ and hence $|\nabla u_K|$ exists as a kind of weak mean curvature on the level surfaces of $u_K$. In fact, if $u=u_K$ then (\ref{e22}) gives that
$$
\begin{cases}
& u(x)=\ln |x|+\ln\hbox{ncap}(K)+ {o}(1)\\
& |\nabla u(x)|=|x|^{-1}\big(1+{o}(1)\big)
\end{cases}
\quad\mbox{as}\quad |x|\to\infty.
$$ 
Now, it follows from
$$
-\hbox{div}(|\nabla u|^{n-2}\nabla u)=0\ \ \hbox{in}\ \ \mathbb R^n\setminus K\ \ \&\ \ |\nabla u|_{\partial K}=c
$$
that if
$$
X=n(x\cdot\nabla u)|\nabla u|^{n-2}\nabla u-|\nabla u|^n x,
$$
$\nu$ stands for the outer unit normal vector, and $r\to\infty$, then
\begin{eqnarray*}
(n-1)nc^nV(K)&=&(1-n)\int_{\partial K}x\cdot \nabla u|\nabla u|^{n-1}\, dS\\
&=&\Big(\frac{n-1}{n}\Big)\int_{\partial K}X\cdot\nu\,dS\\
&=&\Big(\frac{n-1}{n}\Big)\int_{\partial r\mathbb B^n}X\cdot \nu\, dS\\
&=&(n-1)\sigma_{n-1}+{o}(1),
\end{eqnarray*}
and hence $c=\Big({\omega_n}/{V(K)}\Big)^\frac1n.$
\end{remark}

\subsection{Conformal capacity to diameter}\label{s33} To completely reach (\ref{e34}) we establish the following result whose (\ref{e36}) in the case $n=2$ is known (cf. \cite[Theorem 5.3.4]{Ran}).

\begin{theorem}\label{t32} Let $K$ be a compact subset of $\mathbb R^n$. Then
\begin{equation}\label{e36}
\hbox{ncap}(K)\le\frac{\hbox{diam}(K)}{2}
\end{equation}
with equality if $K$ is a ball.
\end{theorem}
\begin{proof} Clearly, if $K$ is a ball then (\ref{e35}) obtains its equality. So, it remains to verify (\ref{e36}). To do so, set $\hbox{dist}(x,K)=\inf_{y\in K}|x-y|$ and
$$
S(K,t)=S\big(\{x\in r\mathbb B^n\setminus K: \hbox{dist}(x,K)=t\}\big)\quad\forall\quad t>0.
$$
A special form of Gehring's \cite[Theorem 2]{Ge2} gives that if
$$
K\subseteq r\mathbb B^n\quad\&\quad T=\liminf_{x\to \mathbb R^n\setminus r\mathbb B^n}\hbox{dist}(x,K),
$$
then
$$
\hbox{ncap}(r\mathbb B^n,K)\le\left(\int_0^T \big(S(K,t)\big)^\frac{1}{1-n}\,dt\right)^{1-n},
$$
and hence
$$
\hbox{nmod}(r\mathbb B^n,K)\ge \int_0^T \Big(\frac{S(K,t)}{\sigma_{n-1}}\Big)^\frac{1}{1-n}\,dt.
$$
This derives

\begin{equation}\label{e37}
\hbox{ncap}(K)\le\exp\left(\lim_{r\to\infty}\Big(\ln r-\int_0^r \Big(\frac{S(K,t)}{\sigma_{n-1}}\Big)^\frac{1}{1-n}\,\Big)dt\right).
\end{equation}

Note that if $\hat{K}$ is the convex hull of $K$ then
$$
\hbox{ncap}(K)\le\hbox{ncap}(\hat{K})\quad\&\quad\hbox{diam}(K)=\hbox{diam}(\hat{K}).
$$
So, we may assume that $K$ is convex, and recall that Kubato's inequality (cf. \cite{Kub} \& \cite{Gri}):
$$
\frac{S(K)}{\sigma_{n-1}}\le\left(\frac{\hbox{diam}(K)}{2}\right)^{n-1}.
$$
Consequently,
$$
\frac{S(K,t)}{\sigma_{n-1}}\le\left(\frac{\hbox{diam}(K)+2t}{2}\right)^{n-1}.
$$
This gives
$$
\ln r-\int_0^r\Big(\frac{S(K,t)}{\sigma_{n-1}}\Big)^\frac1{1-n}\,dt\le\ln\frac{r}{r+\frac{\hbox{diam}(K)}{2}}+\ln\frac{\hbox{diam}(K)}{2}
$$
and so that via (\ref{e37}) one reaches (\ref{e36}).

\end{proof}

\subsection{Conformal capacity to mean-width}\label{s34} Given a nonempty, convex, compact set $K\subseteq\mathbb R^n$. Following \cite[(1.7)]{Schn}, we say that
$$
h_K(x)=\sup_{y\in K}x\cdot y\quad\hbox{for}\quad x\in\mathbb R^n
$$
is the support function of $K$, and
$$
b(K)=\frac{2}{\sigma_{n-1}}\int_{\mathbb S^{n-1}}h_K\,d\theta
$$
is the mean width of $K$ -- here and henceforth $d\theta$ is the uniform surface area measure on $\mathbb S^{n-1}$, i.e., the $n-1$ dimensional spherical Lebesgue measure. The sharp iso-mean-width (or Uryasohn's) inequality 
\begin{equation}
\label{eISOmw}
\left(\frac{V(K)}{\omega_n}\right)^\frac1n\le\frac{b(K)}{2}
\end{equation}
is well known for any compact convex $K\subseteq\mathbb R^n$ (cf. \cite[(6.25)]{Schn}). According to \cite[page 318]{Schn}, we see that if $n=2$ then $b(K)=S(K)/\pi$ and hence (\ref{e39}) has the following replacement which, along with (\ref{e35}), improves (\ref{eISOmw}). 

\begin{theorem}\label{t33} Let $K$ be a compact convex subset of $\mathbb R^n$. Then
\begin{equation}\label{e310}
\hbox{ncap}(K)\le \frac{b(K)}{2}
\end{equation}
with equality if  $K$ is a ball.
\end{theorem}
\begin{proof} Equality of (\ref{e310}) follows from a direct computation with a ball. In general, the argument for (\ref{e310}) is motivated by Borell's proof of the case $n=2$ in \cite[Example 7.4]{Bor}. For $x\in\mathbb R^n$, we have
\begin{equation}\label{e311}
\frac{|x|b(K)}{2}=\frac{1}{\sigma_{n-1}}\int_{\mathbb S^{n-1}}h_K(|x|\theta)\,d\theta.
\end{equation}
The right side of (\ref{e311}) can be approximated by
$\sum_{k=1}^m h_K(|x|\theta_k)\lambda_k$ -- the support function of $\sum_{k=1}^m \lambda_k T_k K$, where
$\lambda_k\in (0,1)$, $\sum_{k=1}^m \lambda_k=1$,
and $T_k K$ is a rotation of $K$ generated by $\theta_k$. Meanwhile, according to Colesanti-Cuoghi's \cite[Theorem 3.1]{ColC} (see also the beginning of Section \ref{s5}), we have
\begin{equation}\label{e312}
\hbox{ncap}\Big(\sum_{k=1}^m \lambda_k T_k K\Big)\ge \sum_{k=1}^m\lambda_k\hbox{ncap}(T_k K)=\hbox{ncap}(K)
\end{equation}
due to the easily-checked rotation-invariance of $\hbox{ncap}(\cdot)$. Note also that the left side of (\ref{e311}) is the support function of a ball of radius $\frac{b(K)}{2}$. So, the above approximation, the correspondence between a support function and a convex set, and (\ref{e312}) yield (\ref{e310}).
\end{proof}

\section{Conformal capacities by mean curvature, surface area and ADM mass}\label{s4a}
\setcounter{equation}{0}

\subsection{Conformal capacity to mean curvature}\label{s4a1} For a convex domain $K\subseteq\mathbb R^n$ with $C^2$ boundary $\partial K$, the $k$-th mean curvature $m_k(K,x)$ at $x\in\partial K$ is defined by 
$$
m_k(K,x)=\begin{cases} & 1\quad\hbox{for}\quad k=0;\\
&{\Big(\begin{array}{c} n-1\\ k\end{array}\Big)^{-1}} {\sum_{1\le i_1<...<i_k\le n-1}\kappa_{i_1}(x)\cdots\kappa_{i_k}(x)}\ \ \hbox{for}\ \ k=1,...,n-1,
\end{cases}
$$
where $\kappa_1(x),...,\kappa_{n-1}(x)$ are the principal curvature of $\partial K$ at the point $x$. Note that (see, e.g. \cite{Spr, FreS})
$$
\begin{cases}
m_1(K,x)=H(K,x)=\hbox{mean\ curvature\ of}\ \partial K\ \hbox{at}\ x;\\
m_2(K,x)=R(K,x)=\hbox{scalar\ curvature\ of}\ \partial K\ \hbox{at}\ x;\\
m_{n-1}(K,x)=G(K,x)=\hbox{Gauss\ curvature\ of}\ \partial K\ \hbox{at}\ x;\\
m_k(K,x)\le \big(H(K,x)\big)^k\quad\hbox{for}\quad k=1,...,n-1;
\end{cases}
$$
Moreover, the $k$-th integral mean curvature of $\partial K$ is given by
$$
M_k(K)=\int_{\partial K}m_k(K,\cdot)\,dS(\cdot).
$$

\begin{theorem}\label{t4a1} If $K\subseteq\mathbb R^n$ is a convex domain with $C^2$ boundary $\partial K$, then
\begin{equation}\label{e38}
\hbox{ncap}(K)\le
\exp\left(\lim_{r\to\infty}\Big(\ln r-\int_0^r \Big(\int_{\partial K}\frac{\big(1+tH(K,\cdot)\big)^{n-1}}{\sigma_{n-1}}\,dS(\cdot)\Big)^\frac1{1-n}\,dt\Big)\right),
\end{equation}
with equality if $K$ is a ball.
\end{theorem}
\begin{proof} A straightforward calculation with $r\mathbb B_n$ yields equality of (\ref{e38}). Now, according to \cite[(13.43)]{San} we have
$$
S(K,t)=\sum_{k=0}^{n-1}\Big(\begin{array}{c} n-1\\ k\end{array}\Big)M_{k}(K)t^k.
$$
This, along with (\ref{e37}) and the binomial formula, deduces 

\begin{eqnarray*}\label{e388}
&&\hbox{ncap}(K)\\
&&\le\exp\left(\lim_{r\to\infty}\Big(\ln r-\int_0^r \Big(\sum_{k=0}^{n-1}\Big(\begin{array}{c} n-1\\ k\end{array}\Big)\Big(\frac{M_{k}(K)}{\sigma_{n-1}}\Big) t^k\Big)^\frac1{1-n}\,dt\Big)\right)\\
&&\le\exp\left(\lim_{r\to\infty}\Big(\ln r-\int_0^r \Big(\sum_{k=0}^{n-1}\Big(\begin{array}{c} n-1\\ k\end{array}\Big)
t^k\int_{\partial K}\Big(\frac{m_k(K,\cdot)}{\sigma_{n-1}}\Big)\,dS(\cdot)\Big)^\frac1{1-n}\,dt\Big)\right)\\
&&\le\exp\left(\lim_{r\to\infty}\Big(\ln r-\int_0^r \Big(\sum_{k=0}^{n-1}\Big(\begin{array}{c} n-1\\ k\end{array}\Big)
t^k\int_{\partial K}\Big(\frac{\big(H(K,\cdot)\big)^k}{\sigma_{n-1}}\Big)\,dS(\cdot)\Big)^\frac1{1-n}\,dt\Big)\right)\\
&&=\exp\left(\lim_{r\to\infty}\Big(\ln r-\int_0^r \Big(\int_{\partial K}\frac{\big(1+tH(K,\cdot)\big)^{n-1}}{\sigma_{n-1}}\,dS(\cdot)\Big)^\frac1{1-n}\,dt\Big)\right),
\end{eqnarray*}
as desired.
\end{proof}

\begin{remark}
\label{r4a1} Here, it is perhaps appropriate to mention that if $n=2$ then the Gauss-Bonnet formula gives $M_1(K)=2\pi$. Now, (\ref{e37}), together with $S(K,t)=S(K)+tM_1(K)$, derives the P\'olya-Szeg\"o inequality (see also
\cite[Aufg. 124]{PolS}, \cite[page 13]{Ge2} and \cite{Bor}):

\begin{equation}\label{e39}
\hbox{ncap}(K)\le \Big(\frac{S(K)}{\sigma_{n-1}}\Big)^\frac1{n-1}\quad\hbox{for}\quad n=2.
\end{equation}
Obviously, (\ref{e39}) and (\ref{e35}) indicate that (\ref{e32}) or the left inequality of (\ref{e33}) under $n=2$ may be further improved. It is our conjecture that (\ref{e39}) is still true for $n\ge 3$ and $K$ being convex.
\end{remark}

\subsection{Surface area to conformal capacity to graphic ADM mass}\label{s4a2}

For a smooth function $f(x)=f(x_1,...,x_n)$ and $i,j,k=1,2,...,n$  we follow \cite{Lam} to write
$$
\begin{cases}
f_i=\frac{\partial f}{\partial x_i};\\
f_{ij}=\frac{\partial^2 f}{\partial x_i\partial x_j};\\
f_{ijk}=\frac{\partial^3 f}{\partial x_i\partial x_j\partial x_k};\\
\delta_{ij}=0\ \hbox{or}\ 1 \ \hbox{as}\ i\not=j\ \hbox{or}\ i=j.
\end{cases}
$$
And, for a bounded open set $O\subseteq\mathbb R^n$ with $n\ge 3$ and boundary $\partial O$, we say that a smooth function $f:\mathbb R^n\setminus O\mapsto \mathbb R^1$ is asymptotically flat if 
$$
|f_i(x)|+|x||f_{ij}(x)|+|x|^2|f_{ijk}(x)|=\mathcal{O}(|x|^{-\gamma/2})
\quad\hbox{for}\quad |x|\to\infty.
$$
holds for a constant $\gamma>n/2-1$. Then, given such a smooth asymptotically flat function $f$, let 
$$
\big(\mathbb R^n\setminus O, \delta+df\otimes df\big)=\big(\mathbb R^n\setminus O, (\delta_{ij}+f_if_j)\big)
$$ 
be the graph of $f$, which is actually a complete Riemannian manifold. Now, the Arnowitt-Deser-Misner (ADM) mass of such a graph is defined by
$$
m_{ADM}(\mathbb R^n\setminus O,\delta+df\otimes df)=\lim_{r\to\infty}\int_{S_r}\sum_{i,j=1}^n\frac{(f_{ii}f_j-f_{ij}f_i)x_j|x|^{-1}}{2(n-1)\sigma_{n-1}(1+|\nabla f|^2)}\,d\sigma,
$$
where $S_r$ is the coordinate sphere of radius $r$ and $d\sigma$ is the area element of $S_r$. Here, it is worth mentioning to point out the above-defined ADM mass is the same as the original ADM mass of an asymptotically flat manifold; see Schoen-Yau \cite{SY1,SY2} and Witten \cite{Wi} for the Riemannian positive mass theorem, as well as Huisken-Illmanen \cite{HI1} and Bray \cite{Br} for the Riemannian Penrose inequlaity for area outer minimizing horizon. 

\begin{theorem}\label{t4a2} Given a convex compact set $K\subseteq\mathbb R^n$, $n\ge 3$ and a positive constant $c$, let $O\subseteq\mathbb R^n$ and $u$ be the convex domain containing $K$ and the $n$-capacity potential that solve the exterior Bernoulli problem below
\begin{equation}
\label{eee}
\begin{cases}
-\hbox{div}(|\nabla u|^{n-2}\nabla u)=0\quad\hbox{in}\quad O\setminus K;\\
u=1\quad\hbox{on}\quad \partial K;\\
u=0\quad\hbox{on}\quad \partial O;\\
|\nabla u|=c\quad\hbox{on}\quad \partial O.\\
\end{cases}
\end{equation}
Suppose $f: \mathbb R^n\setminus O\mapsto \mathbb R^1$ is a smooth asymptotically flat function such that $f(\partial O)$ is in a level set of $f$, $\lim_{x\to\partial O}|\nabla f(x)|=\infty$, and the scalar curvature of $\big(\mathbb R^n\setminus O, \delta+df\otimes df\big)$ is non-negative. Then
\begin{equation}
\label{eADMc}
 m_{ADM}(\mathbb R^n\setminus O,\delta+df\otimes df)\ge
 \frac{\big(\frac{S(O)}{\sigma_{n-1}}\big)^\frac{n-2}{n-1}-\big(\frac{S(K)}{\sigma_{n-1}}\big)^\frac{n-2}{n-1}}
{2(n-2)\Big(\frac{\hbox{ncap}(K,O)}{\sigma_{n-1}}\Big)^\frac1{1-n}}.
\end{equation}
\end{theorem}

\begin{proof} First of all, it should be pointed out that (\ref{eee}) is solvable; see also \cite{HS}.

Next, Lam's \cite[Theorem 6]{Lam} actually says  
\begin{equation}
\label{eLam}
m_{ADM}(\mathbb R^n\setminus O,\delta+df\otimes df)
=\int_{\partial O}\frac{H(O,\cdot)}{2\sigma_{n-1}}\,dS(\cdot)+\int_{\mathbb R^n\setminus O}\frac{R_f(\cdot)}{2(n-1)\sigma_{n-1}}\,dV(\cdot),
\end{equation}
where 
$$
R_f=\sum_{j=1}^n \frac{\partial}{\partial x_j}\sum_{i=1}^n\left(\frac{f_{ii}f_j-f_{ij}f_i}{1+|\nabla f|^2}\right)
$$
is the scalar curvature of the graph $(\mathbb R^n\setminus O, \delta+df\otimes df)$ of $f$; see also \cite[Lemma 10]{Lam} or \cite[Proposition 5.4]{HuangW}. Thus, (\ref{eLam}), along with $R_f\ge 0$ implies 
\begin{equation}
\label{eL}
({2\sigma_{n-1}})^{-1}\int_{\partial O}H(O,\cdot)\,dS(\cdot)
\le m_{ADM}(\mathbb R^n\setminus O,\delta+df\otimes df).
\end{equation}

Finally, in accordance with \cite[(4.9)\&(4.28)]{PhiP} we have
\begin{equation}
\label{eLL}
\big(S(K)\big)^\frac{n-2}{n-1}-\big(S(O)\big)^\frac{n-2}{n-1}\ge(2-n)\big(S(O)\big)^{\frac1{n-1}}\int_{\partial O}H(O,\cdot)|\nabla u|^{-1}\,dS(\cdot).
\end{equation}
Note that $|\nabla u|$ equals a constant $c>0$ on $\partial O$. So, an application of (\ref{e13}) with $t=0$ gives
$$
c^{n-1}S(O)=\hbox{ncap}(O,K).
$$
This, together with (\ref{eLL})-(\ref{eL}), implies (\ref{eADMc}) right away.
\end{proof}

\begin{remark}\label{r4ab} Two comments are in order.

\item{\rm(i)} Under the same hypothesis as in Theorem \ref{t4a2}, we find that the right-hand-side of \cite[(4.28)]{PhiP} is non-negative, and then utilize \cite[(4.8)]{PhiP} and  \cite[Lemma 12]{Lam} to derive 
$$
\sigma_{n-1}\Big(\frac{S(O)}{\sigma_{n-1}}\Big)^\frac{n-2}{n-1}\le\int_{\partial O}H(O,\cdot)\,dS(\cdot)\le \Big(\frac{S(O)}{n-2}\Big)\left(\frac{\hbox{ncap}(K,O)}{S(O)}\right)^\frac1{n-1}.
$$
and thus
$$
n-2\le\Big(\frac{\hbox{ncap}(K,O)}{\sigma_{n-1}}\Big)^\frac1{n-1}.
$$
This and (\ref{eADMc}) imply
$$
m_{ADM}(\mathbb R^n\setminus O,\delta+df\otimes df)\ge
 2^{-1}\left({\Big(\frac{S(O)}{\sigma_{n-1}}\Big)^\frac{n-2}{n-1}-\Big(\frac{S(K)}{\sigma_{n-1}}\Big)^\frac{n-2}{n-1}}\right).
$$ 
Upon $K$ shrinking to a point, the last inequality recovers the following Riemannian Penrose type inequality (established in \cite[Remark 8]{Lam}):
 \begin{equation*}\label{eADMs}
2^{-1}\left(\frac{S(O)}{\sigma_{n-1}}\right)^\frac{n-2}{n-1}\le
m_{ADM}(\mathbb R^n\setminus O,\delta+df\otimes df).
\end{equation*}

\item{\rm(ii)} Moreover, if $S(t)$ and $V(t)$ stand for the surface area of the level surface $\Gamma_t=\{x\in \overline{O\setminus K}:\ u(x)=t\}$ and the volume of the domain bounded by $\Gamma_t$, then an application of the co-area formula, (\ref{e13}) and the H\"older inequality yields
$$
S(t)\le\big(\hbox{ncap}(K,O)\big)^\frac1n\big(-V'(t)\big)^\frac{n-1}{n}\quad\forall\quad t\in [0,1].
$$
According to \cite[Theorem 4]{PhiP}, we have
$$
S''(t)S(t)-(n-1)^{-1}\big(S'(t)\big)^2\ge 0\quad\forall\quad t\in [0,1],
$$
thereby getting
$$
\big(S(t))^{\frac1{1-n}}S'(t)\ge S'(0)\big(S(0)\big)^\frac1{1-n}.
$$
Using $|\nabla u|\big|_{\partial O}=c$ and \cite[4.8]{PhiP} again, we find
$$
S'(0)\big(S(0)\big)^\frac1{1-n}=-(n-1)\big(\hbox{ncap}(K,O)\big)^\frac1{1-n}\int_{\partial O}H(O,\cdot)\,dS(\cdot)
$$
whence achieving the monotonicity involving volume, surface area and mean curvature below:
$$
\frac{d}{dt}\left(\big(S(t)\big)^2-2(n-1)V(t)\int_{\partial O}H(O,\cdot)\,dS(\cdot) \right)\ge 0\quad\forall\quad t\in [0,1],
$$
where the special case $n=2$ goes back to Longinetti's isoperimetric deficit monotonicity in \cite[(5.12)]{Log}.
\end{remark}

\section{Boundary estimates for gradients of $n$-equilibriums}\label{s4}
\setcounter{equation}{0}

\subsection{An identity for the unit sphere area via $n$-equilibrium}\label{s41} By a convex body in $\mathbb R^n$ we mean a convex and compact subset of $\mathbb R^n$ with non-empty interior. For convenience, denote by $\mathbb K^n$ the set of all convex bodies. For $K\in\mathbb K^n$, the Gauss map $g: \partial K\to \mathbb S^{n-1}$ is defined almost everywhere with respect to surface measure $dS$ and determined by $g(x)=\nu$, the outer unit normal at $x\in \partial K$.
In the process of finding a representation of the conformal capacity $\hbox{ncap}(K)$ in terms of the integral of $|\nabla u_K|^n$ of $n$-equilibrium $u_K$ on $\partial K$, we get the following result whose case $n=2$ is essentially known; see also \cite{Je96b}.

\begin{theorem}\label{p-41} If $K\in\mathbb K^n$, then
\begin{equation}\label{e42a}
\int_{\partial K}h_{K}(g)|\nabla u_{K}|^n \,dS=\sigma_{n-1}.
\end{equation}
In other words, if $g_\ast(|\nabla u_{K}|^n \,dS)$ is defined by
$$
\int_{g^{-1}(E)}|\nabla u_{K}|^n \,dS\quad\forall\quad \hbox{Borel\ set}\ \ E\subseteq\mathbb S^{n-1},
$$
then
$$
\int_{\mathbb S^{n-1}}h_K g_\ast(|\nabla u_{K}|^n \,dS)=\sigma_{n-1}.
$$
Consequently,
\begin{equation}\label{e42b1}
\int_{\mathbb S^{n-1}}\xi g_\ast(|\nabla u_{K}|^n \,dS)(\xi)=0.
\end{equation}

\end{theorem}

\begin{proof} For $K\in\mathbb K^n$, write $u=u_{K}$. Suppose $\nu$ is the outer unit normal. Two cases are in order.

{\it Case 1.}\ $K$ is of $C^2$ strictly convex. Then
\begin{equation}\label{e44}
|\nabla u|=-\frac{\partial u}{\partial\nu}\quad\hbox{on}\quad \partial K;
\end{equation}
see also \cite{Sak}.

Recall that if
$$
X=n(x\cdot\nabla u)|\nabla u|^{n-2}\nabla u-|\nabla u|^n x
$$
then $\hbox{div}X=0$ in $\mathbb R^n\setminus K$ and hence by an integration-by-part,
$$
\int_{\partial K}X\cdot\nu\,dS=\int_{\partial (r\mathbb B^n)}X\cdot\nu\,dS\quad\hbox{as}\quad r\to\infty.
$$
However, the right side of the last formula tends to $\sigma_{n-1}$ as $r\to\infty$ thanks to the expansion of $u$ at infinity. So, from (\ref{e44}) it follows that
$$
(n-1)\int_{\partial K}(x\cdot\nabla u)\Big(-\frac{\partial u}{\partial\nu}\Big)^{n-1}\,dS=\Big(\frac{1-n}{n}\Big)\int_{\partial K}X\cdot\nu\,dS=(n-1)\sigma_{n-1}.
$$
Consequently, (\ref{e42a}) follows from
$$
\int_{\partial K}h_{K}(g)|\nabla u|^n\,dS=\int_{\partial K}(x\cdot\nabla u)\Big(-\frac{\partial u}{\partial\nu}\Big)^{n-1}\,dS=\sigma_{n-1}.
$$

To reach (\ref{e42b1}), note that $\sigma_{n-1}$ is a dimensional constant and the support function of $L=K+x_0$ is
$$
h_L(\xi)=h_K(\xi)+x_0\cdot\xi\quad\hbox{for}\quad\xi\in\mathbb S^{n-1},
$$
where $x_0\in\mathbb R^n$ is arbitrarily given. So, an application of (\ref{e42a}) to $L$ yields
$$
\int_{\partial K}x_0\cdot g(x)|\nabla u_K(x)|^n\,dS(x)=0
$$
and consequently, the following vector equation
$$
\int_{\partial K}g(x)|\nabla u_K(x)|^n\,dS(x)=0
$$
holds. This gives (\ref{e42b1}).

{\it Case 2}.\ $K$ just belongs to $\mathbb K^n$. To prove (\ref{e43b}) under this general situation, recall first that the Hausdorff metric $d_H$ on the class $\mathcal C^n$ of all compact convex subsets of $\mathbb R^n$ is determined by
$$
d_H(K_1,K_2)=\sup_{x\in K_1}d(x,K_1)+\sup_{x\in K_2}d(x,K_2)\quad\forall\quad K_1,K_2\in \mathcal C^n,
$$
where $d(x,E)$ stands for the distance from the point $x$ to the set $E$.

Of course, the interior of the above $K$ is a Lipschitz domain. According to Lewis-Nystr{\"o}m's
\cite[Theorem 3]{LewN} (cf. \cite{Dah} and \cite{JeK} for harmonic functions), we see that $\nabla u_K$ has non-tangential limit, still denoted by $\nabla u_K$, almost everywhere on $\partial K$ with respect to $dS$. Moreover, $|\nabla u_K|$ is $n$-integrable on $\partial K$ under $dS$, i.e.,
\begin{equation}\label{e43c}
\int_{\partial K}|\nabla u_K|^n\,dS<\infty.
\end{equation}
For $0<t<1$ let
$$
L_t=\{x\in \mathbb R^n\setminus K:\ u_K(x)>t\}\quad\&\quad K_t=\mathbb R^n\setminus L_t.
$$
Then $K_t$ is $C^2$ strictly convex (cf. \cite[Theorem 2.2]{ColC}). Note that $u_K-t$ is equal to the $n$-equilibrium potential $u_{K_t}$ of $K_t$, and note that continuity of $u_K$ on $\partial K$ yields $\lim_{t\to 0}d_H(K_t,K)=0$. So,
$$
\sigma_{n-1}=\int_{\partial K_t}(x\cdot\nabla u_K)|\nabla u_K|^{n-1}\,dS(x).
$$
This, plus (\ref{e43c}) and the dominated convergence theorem, derives
$$
\sigma_{n-1}=\lim_{t\to 0}\int_{\partial K_t}(x\cdot\nabla u_K)|\nabla u_K|^{n-1}\,dS(x)=\int_{\partial K}(x\cdot\nabla u_K)|\nabla u_K|^{n-1}\,dS(x),
$$
whence yielding (\ref{e43a}) and its consequence (\ref{e42b1}).
\end{proof}

\begin{remark}\label{rrr} Given $K\in\mathbb K^n$. If $U$ is $n$-harmonic, i.e., $\hbox{div}(|\nabla U|^{n-2}\nabla U)=0$, in $\mathbb R^n\setminus K$, $U$ is continuous on $\partial K$, and $U(x)$ has a finite limit $U(\infty)$ as $x\to \infty$, then the well-known divergence theorem is used to produce
$$
U(\infty)=\frac1{\sigma_{n-1}}\left(\int_{\partial K}U|\nabla u_K|^{n-1}\,dS+\int_{\mathbb R^n\setminus K}\frac{\nabla u_K\cdot\nabla U}{\big(|\nabla u_K|^{n-2}-|\nabla U|^{n-2}\big)^{-1}}\,dV\right).
$$
In particular, if $n=2$ then this formula reduces to \cite[(6.3)]{Je96b}, and consequently, if $U(x)=u_K(x)-\ln|x|$ (which is $2=n$-harmonic in $\mathbb R^n\setminus K$) then
\begin{equation*}\label{e42b11}
\hbox{ncap}(K)=\exp\left(\frac{1}{\sigma_{n-1}}\int_{\partial K}(\ln|x|)|\nabla u_K|^{n-1}\,dS\right)\ \ \hbox{for}\ \ n=2.
\end{equation*}
It is our conjecture that this last formula is still valid for $n\ge 3$.
\end{remark}

 \subsection{A lower bound for the gradient of $n$-equilibrium}\label{s42}  Being motivated by \cite[Lemma 2.18]{CLNSXYZ} we find the following lower bound estimate for the gradient of the equilibrium of (\ref{e16}) on the boundary of a convex body.

\begin{theorem}\label{l5c} Given $K\in\mathbb K^n$, let $u_K$ be its equilibrium potential. If $K\subseteq r\mathbb B^n$, then there exists a constant $c>0$ depending only on $r$ and $n$ such that $\inf_{\partial K}|\nabla u_K|\ge c$ almost everywhere on $\partial K$ with respect to $dS$.
\end{theorem}
\begin{proof} Suppose $u=u_K$ and $t_0\in (0,1)$ obey
$$
K_t=\{x\in\mathbb R^n\setminus{K}:\ u(x)\le t\}\subseteq r\mathbb B^n\quad\forall\quad t\in (0,t_0).
$$
Note that $K_t$ is $C^2$ strictly convex and the existence of $t_0$ is ensured by the continuity of $u$ in $\mathbb R^n\setminus K$ (cf. \cite[Theorem 2.2]{ColC}). Now, for $t\in (0,t_0)$ let
$$
\check{u}_t(x)={u(x)}-{t}\quad\forall\quad x\in\mathbb R^n\setminus K_t.
$$
Then $\check{u}_t$ is the solution of (\ref{e16}) for $K_t$, and in $C^2(\mathbb R^n\setminus K_t)$. For $\tau\in [0,1)$ let $\check{K}_\tau=\{x\in \mathbb R^n\setminus K_t:\ \check{u}_t(x)\le \tau\}$ and
$h(\cdot,\tau)$ be its support function $h_{\check{K}_\tau}$. Since $\check{K}_0={K_t}\subseteq r\mathbb B^n$, $\check{u}_t$ is controlled, via the maximum principle, by the $n$-equilibrium potential of $r\mathbb B^n$. Consequently, there is a constant $c_0>0$ depending on $n$ and $r$ such that
$$
\hbox{diam}(\check{K}_{2^{-1}})=\hbox{diam}(\{x\in\mathbb R^n: 2^{-1}<u(x)\le 1\})\le c_0.
$$
Moreover, we have
$$0\le \inf_{x\in\mathbb S^{n-1}}h(x,2^{-1})\le\sup_{x\in\mathbb S^{n-1}}h(x,2^{-1})\le c_0,$$
whence deriving
$$
h(x,0)=h(x,2^{-1})-\int_{0}^{2^{-1}}\frac{\partial h}{\partial \tau}(x,\tau)\,d\tau\quad\forall\quad x\in\mathbb S^{n-1}.
$$
From \cite[Theorem A.2]{ColC} it follows that $s\mapsto\frac{\partial h}{\partial \tau}(x,\tau)$ is a non-decreasing function on $[0,1)$. This monotonicity and the mean-value theorem for derivatives yield
$$
\frac{\partial h}{\partial \tau}(x,\tau)\Big|_{\tau=0}\le 2\big(h(x,2^{-1})-h(x,0)\big)\le 2h(x,2^{-1})\le 2c_0\ \ \forall\ \ x\in\mathbb S^{n-1}.
$$
Meanwhile, an application of \cite[Theorem A.1]{ColC} gives
$$
\frac{\partial h}{\partial \tau}(x,\tau)|_{\tau=0}={|\nabla{\check u}_t(x)|^{-1}},
$$
where $x\in\partial K_t$ satisfies 
$$
x={(\nabla\check{u}_t(x))}{|\nabla\check{u}_t(x)|^{-1}}\ \  \&\ \ \check{u}_t(x)=0.
$$
As a result, we get
$$
\inf_{x\in\partial K_t}|\nabla{u}(x)|=\inf_{x\in\partial K_t}|\nabla\check{u}_t(x)|\ge (2c_0)^{-1}.
$$
The desired assertion follows by letting $t\to 0$ and using the existence of the non-tangential maximal function of $|\nabla u|$ on $\partial K$.
\end{proof}

\section{Hadamard's variation for conformal capacities}\label{s5}
\setcounter{equation}{0}

\subsection{Hadamard's variation: the smooth case}\label{s51} For $K_1,K_2\in\mathbb K^n$ and $0\le t_1,t_2$ define
$$
t_1K_1+t_2K_2=\{x=t_1x_1+t_2x_2:\ x_j\in K_j\}.
$$
In accordance with Colesant-Cuoghi's \cite[Theorem 3.1]{ColC} (cf. Borell \cite{Bor} for $n=2$), we have the following Brunn-Minkowski inequality for $t\in [0,1]$ and $K_1,K_2\in\mathbb K^n$:

\begin{equation}\label{e41}
\hbox{ncap}(tK_1+(1-t)K_2)\ge t\hbox{ncap}(K_1)+(1-t)\hbox{ncap}(K_2)
\end{equation}
with equality if and only if $K_1$ is a translate and a dilate of $K_2$.

Notice that (\ref{e41}) implies that
$$
\frac{d^2}{dt^2}\hbox{ncap}(tK_1+(1-t)K_2)\big|_{t=0}\le 0.
$$
So, we get the following assertion extending the smooth two-dimensional Hadamard's variation formula (cf. \cite{Sch}).

\begin{theorem}\label{t41} If $K_0,K_1\in \mathbb K^n$ are $C^2$ strictly convex, then
\begin{equation}\label{e42b}
\frac{d}{dt}\ln\hbox{ncap}(K_0+tK_1)\big|_{t=0}={\sigma_{n-1}}^{-1}\int_{\partial K_0}h_{K_1}(g)|\nabla u_{K_0}|^n \,dS,
\end{equation}
equivalently,
\begin{equation}\label{e43a}
\frac{d}{dt}\ln\hbox{ncap}((1-t)K_0+tK_1)\big|_{t=0}={\sigma_{n-1}}^{-1}\int_{\partial K_0}\frac{|\nabla u_{K_0}|^n}{\big(h_{K_1}(g)-h_{K_0}(g)\big)^{-1}} \,dS.
\end{equation}
Consequently,
\begin{equation}\label{e43b}
\frac{\sigma_{n-1}}{\hbox{ncap}(K_0)}\le\int_{\partial K_0}|\nabla u_{K_0}|^n\,dS
\end{equation}
with equality if $K_0$ is a ball.
\end{theorem}
\begin{proof} To derive (\ref{e42b}), note again that
$$
u(x)=\ln |x|-\ln\hbox{ncap}(K)+{o}(1)\quad \forall\ \ x\in \mathbb R^n\setminus K.
$$
Proving (\ref{e42b}) is equivalent to establishing the first variation of $u$. To do so, for an arbitrary small number $\epsilon>0$ let $K_\epsilon$ be such a convex body that its boundary $\partial K_\epsilon$ is obtained by shifting $\partial K$ an infinitesimal distance $\delta\nu=\epsilon\rho(s)$ along its outer unit normal $\nu$, where $\rho$ is a smooth function on $\partial K$:
$$
\partial K_\epsilon=\{x+\epsilon\rho(x)\nu(x):\ x\in\partial K\}.
$$
and denote by $u_\epsilon=u_{K_\epsilon}$.

For convenience, set $K^c=\mathbb R^n\setminus K$, $K_\epsilon^c=\mathbb R^n\setminus K_\epsilon$, and define $u(x)=0$ for $x\in K$ and $u_\epsilon(x)=0$ for $x\in K_\epsilon$. Consider the following difference
\begin{equation}\label{e46}
\hbox{Dif}(\epsilon)=\int_{K^c}|\nabla u|^{n-2}\nabla u\cdot\nabla u_\epsilon\,dV-\int_{K_\epsilon^c}|\nabla u_\epsilon|^{n-2}\nabla u_\epsilon\cdot\nabla u\,dV.
\end{equation}

On the one hand,
\begin{eqnarray*}
\hbox{Dif}(\epsilon)&&=\int_{K^c\setminus K_\epsilon^c}|\nabla u|^{n-2}\nabla u\cdot\nabla u_\epsilon\,dV\\
&&\ +\int_{K_\epsilon^c}(|\nabla u|^{n-2}-|\nabla u_\epsilon|^{n-2})\nabla u_\epsilon\cdot\nabla u\,dV\\
&&=\epsilon\int_{\partial K^c}|\nabla u|^{n-1}\Big(\frac{\partial u_\epsilon}{\partial\nu}\Big)\rho\,dS\\
&&\ +\int_{K_\epsilon^c}(|\nabla u|^{n-2}-|\nabla u_\epsilon|^{n-2})\nabla u_\epsilon\cdot\nabla u\,dV.\\
\end{eqnarray*}
This yields
$$
\lim_{\epsilon\to 0}\frac{\hbox{Dif}(\epsilon)}{\epsilon}=-\int_{\partial K}|\nabla u|^{n-1}\Big(\frac{\partial u}{\partial\nu}\Big)\rho\,dS.
$$

On the other hand, note that
$$
\begin{cases}
-\hbox{div}(|\nabla u_\epsilon|^{n-2}\nabla u_\epsilon=0\quad\hbox{in}\quad K^c_\epsilon;\\
-\hbox{div}(|\nabla u_\epsilon|^{n-2}\nabla u_\epsilon=0\quad\hbox{in}\quad K^c,
\end{cases}
$$
and
$$
\begin{cases}
\hbox{div}(u|\nabla u_\epsilon|^{n-2}\nabla u_\epsilon)=u\hbox{div}\big(|\nabla u_\epsilon|^{n-2}\nabla u_\epsilon\big)+|\nabla u_\epsilon|^{n-2}\nabla u_\epsilon\cdot\nabla u;\\
\hbox{div}(u_\epsilon|\nabla u|^{n-2}\nabla u)=u_\epsilon\hbox{div}\big(|\nabla u|^{n-2}\nabla u\big)+|\nabla u|^{n-2}\nabla u\cdot\nabla u_\epsilon.
\end{cases}
$$
So, an application of the divergence theorem gives
\begin{eqnarray*}
\int_{K^c}|\nabla u|^{n-2}\nabla u\cdot\nabla u_\epsilon\,dV
&&=\int_{K^c}\hbox{div}(u_\epsilon|\nabla u|^{n-2}\nabla u)\,dV\\
&&=\lim_{r\to\infty}\int_{K^c\setminus (r\mathbb B^n)^c}\hbox{div}(u_\epsilon |\nabla u|^{n-2}\nabla u)\,dV\\
&&=\int_{\partial K^c}u_\epsilon|\nabla u|^{n-2}\nabla u\cdot\nu\,dS\\
&&\ -\lim_{r\to\infty}\int_{\partial (r\mathbb B^n)^c}u_\epsilon|\nabla u|^{n-2}\nabla u\cdot\nu\,dS\\
&&=-\lim_{r\to\infty}\int_{\partial (r\mathbb B^n)^c}u_\epsilon|\nabla u|^{n-2}\nabla u\cdot\nu\,dS.
\end{eqnarray*}
Similarly, we have
$$
\int_{K_\epsilon^c}|\nabla u_\epsilon|^{n-2}\nabla u_\epsilon\cdot \nabla u\,dV=-\lim_{r\to\infty}\int_{\partial (r\mathbb B^n)^c}u|\nabla u_\epsilon|^{n-2}\nabla u_\epsilon\cdot\nu\,dS.
$$
Consequently,
\begin{eqnarray*}
\hbox{Dif}(\epsilon)&=&-\lim_{r\to\infty}\left(\int_{\partial (r\mathbb B^n)^c}u_\epsilon|\nabla u|^{n-2}\nabla u\cdot\nu\,dS-\int_{\partial (r\mathbb B^n)^c}u|\nabla u_\epsilon|^{n-2}\nabla u_\epsilon\cdot\nu\,dS\right)\\
&=&-\lim_{r\to\infty}\int_{\partial (r\mathbb
B^n)^c}(u_\epsilon-u)|\nabla u|^{n-2}\nabla
u\cdot\nu\,dS\\
&&\ +\lim_{r\to\infty}\int_{\partial (r\mathbb B^n)^c}u(|\nabla
u_\epsilon|^{n-2}\nabla u_\epsilon-|\nabla u|^{n-2}\nabla
u)\cdot\nu\,dS.
\end{eqnarray*}
This derives via (\ref{e42a})
\begin{eqnarray*}
\lim_{\epsilon\to 0}\frac{\hbox{Dif}(\epsilon)}{\epsilon}&=&\lim_{\epsilon\to 0}\left(\frac{\ln\hbox{ncap}(K_\epsilon)-\ln\hbox{ncap}(K)}{\epsilon}\right)\lim_{r\to\infty}\int_{\partial (r\mathbb B^n)^c}\frac{\nabla u\cdot\nu}{|\nabla u|^{2-n}}\,dS\\
&=&-\sigma_{n-1}\lim_{\epsilon\to
0}\frac{\ln\hbox{ncap}(K_\epsilon)-\ln\hbox{ncap}(K)}{\epsilon}.
\end{eqnarray*}

The above two formulas for $\lim_{\epsilon\to 0}\epsilon^{-1}\hbox{Dif}(\epsilon)$ derive
$$
\lim_{\epsilon\to
0}\frac{\ln\hbox{ncap}(K_\epsilon)-\ln\hbox{ncap}(K)}{\epsilon}
=\frac{1}{\sigma_{n-1}}\int_{\partial K}|\nabla u|^{n}\rho\,dS,
$$
and thereby verifying (\ref{e42b}) through letting $K=K_0$ and $\rho=h_{K_1}\circ g$.

Through the chain rule and the homogeneous property of the support function, (\ref{e42b}) immediately derives (\ref{e43a}) and vice visa. Now, because $t\mapsto\hbox{ncap}\big((1-t)K_0+tK_1\big)$ is concave on $[0,1]$; see also \cite{ColC}, if $K_1=r\mathbb B^n$ and $r=\hbox{ncap}(K_0)$ then an application of (\ref{e43a}) gives
\begin{eqnarray*}
0&\le&\frac{d}{dt}\ln\hbox{ncap}\big((1-t)K_0+tK_1\big)\big|_{t=0}\\
&=&\Big(\frac{1}{\hbox{ncap}(K_0)}\Big)\frac{d}{dt}\hbox{ncap}\big((1-t)K_0+tK_1\big)\big|_{t=0}\\
&=&\frac1{\sigma_{n-1}}\int_{\partial K_0}\big(h_{K_1}(g)-h_{K_0}(g)\big)|\nabla u_{K_0}|^n\,dS\\
&=&\frac{1}{\sigma_{n-1}}\int_{\partial K_0}\big(r-h_{K_0}(g)\big)|\nabla u_{K_0}|^n\,dS,
\end{eqnarray*}
whence reaching (\ref{e43b}) via (\ref{e42a}).
\end{proof}

\subsection{Hadamard's variation: the non-smooth case}\label{s52} To generalize Theorem \ref{t41}, without loss of generality we may assume that the origin is an interior point of $K, K_j\in\mathbb K^n$, write $\varrho_K:\mathbb S^{n-1}\mapsto\partial K$ and $\varrho_{K_j}:\mathbb S^{n-1}\mapsto\partial K_j$ for the radial projections
$$
\mathbb S^{n-1}\ni\theta\mapsto\varrho_K(\theta)=r_K(\theta)\theta\in\partial K
$$
and
$$
\mathbb S^{n-1}\ni\theta\mapsto\varrho_{K_j}(\theta)=r_{K_j}(\theta)\theta\in\partial K_j
$$
respectively, where $r_K(\theta)$ and $r_{K_j}$ are the unique positive numbers ensuring $r_K(\theta)\theta\in\partial K$ and $r_{K_j}(\theta)\theta\in\partial K_j$ respectively, and set
$$
D(\theta)={|\nabla u_K(\varrho_K(\theta))|r_K(\theta)}{(h_K\big(g(\varrho_K(\theta))\big))^{-\frac1n}}
$$
and
$$
D_j(\theta)={|\nabla u_{K_j}(\varrho_{K_j}(\theta))|r_{K_j}(\theta)}{(h_{K_j}\big(g(\varrho_{K_j}(\theta))\big))^{-\frac1n}}
$$
respectively.

In the sequel, we will use the fact that $dS(x)=|x|^n(x\cdot g(x))^{-1}d\theta$ holds for $\theta=x/|x|$.

\begin{theorem}\label{Le} For $\{K, K_1, K_2,...\}\subseteq\mathbb K^n$, $\epsilon>0$ and $\alpha>0$, there exist $s_0>0$, $\eta>0$ and a family of balls $\mathcal B$ on $\mathbb S^{n-1}$ such that:
\begin{itemize}
\item[\rm(i)] every member in $\mathcal B$ has radius $s_0$;
\item[\rm(ii)] there is a constant $N>0$ depending only on the inner and outer radii of $K$, such that any point of $\mathbb S^{n-1}$ belongs to at most $N$ balls of $\mathcal B$;
\item[\rm(iii)] $S(\mathbb S^{n-1}\setminus F)<\epsilon$ where $F=\cup_{B\in\mathcal{B}}B$;
\item[\rm(iv)] if ${d}_H(K_j,K)<\eta$, then for any $B\in\mathcal B$ one has
$$
s_0^{1-n}\left(\int_B\Big|\Big(\frac{D_j(\theta)}{D(\theta)}\Big)^{n-1}-1\Big|^\alpha\,d\theta+
\int_B\Big|\Big(\frac{D(\theta)}{D_j(\theta)}\Big)^{n-1}-1\Big|^\alpha\,d\theta\right)<\epsilon\,;
$$
\item[\rm(v)]
$$
\lim_{j\to\infty}\int_{\mathbb S^{n-1}}\big|D_i^n(\theta)-D^n(\theta)\big|\,d\theta=0\,.
$$
\end{itemize}
\end{theorem}

\begin{proof} According to Jerison's \cite[Lemma 3.3]{Je96a}, we have that
for any $\epsilon>0$ there exists $\eta>0$ and a finite disjoint
collection of open balls $B_{r_k}(z_k)$ (centered at $z_k$ with radius $r_k$) such that $z_k\in\partial K$
and for any convex body $L\in\mathbb K^n$ for which ${d}_H(L,K)<\eta$:
\begin{itemize}
\item[\rm(a)] $S\big(\partial L\setminus\cup_k B_{r_k}(z_k)\big)<\epsilon$;
\item[\rm(b)] after a suitable rotation and translation depending on $k$, one has that $\partial K$ and $\partial L$ are given on $B_{r_k}(z_k)$ by the graphs of functions $\phi$ and $\psi$ respectively, enjoying
$$
\sup \left\{|\nabla \phi(x)|+|\nabla\psi(x)|\,:\,|x|<\epsilon^{-1}{r_k},\
\mbox{$\phi$\ \&\ $\psi$ differentiable at
$x$}\right\}\le\epsilon\,.
$$
\end{itemize}
Now, given $\epsilon>0$. Following the beginning part of the proof of Jerison's \cite[Lemma 3.7]{Je96a} we choose a sufficiently small number $s_0<\min\{r_k\}$ such that the Jacobians of the change of variables $\varrho_{K_k}$ and $\varrho_K$ vary by at most $\epsilon$ as $\theta$ varies by the distance $s>0$ and $\varrho_K(\theta)$ is contained in $\cup_{k}B_{r_k}(z_k)$. As a consequence, we can select $\mathcal{B}$ obeying (i)-(ii)-(iii) described as above.

Meanwhile, from Lewis-Nystr\"om's \cite[Theorem 2]{LewN8} it follows that for each $s\in (0,s_0)$ and each ball $B$ of radius $s$ in the concentric $\epsilon^{-1}$ multiple of any element in $\mathcal{B}$, there is a constant $c_B$ such that
\begin{equation}\label{BMO}
s^{1-n}\int_B|\ln D(\theta)-c_B|\,d\theta<\epsilon.
\end{equation}
Furthermore, using the previously-stated (a)-(b) we can take $\delta>0$ small enough to obtain
\begin{equation}\label{BMOa}
s^{1-n}\int_B|\ln D_j(\theta)-c_B|\,d\theta<\epsilon\quad\forall\quad s\in (0,\delta).
\end{equation}
A combination of (\ref{BMO}) and (\ref{BMOa}) gives
\begin{equation*}\label{BMOb}
s^{1-n}\int_B\Big|\ln\frac{ D_j(\theta)}{D(\theta)}\Big|\,d\theta<2\epsilon.
\end{equation*}
Applying John-Nirenberg's exponential inequality (cf. \cite{JoNi61}) for a BMO-function to (\ref{BMO}), we obtain that given $\alpha>0$ and for arbitrarily small $\epsilon'>0$ one can take $\eta'>0$ and $s_0$ so small that for each $B\in\mathcal B$ there is a constant $c'_B$ ensuring
\begin{equation}\label{eqe5}
s_0^{1-n}\int_B\Big|c'_B\Big(\frac{D_j(\theta)}{D(\theta)}\Big)^{n-1}-1\Big|^\alpha\,d\theta<\epsilon'.
\end{equation}
 Note that $\eta'$ and $s_0$ can be chosen small enough to ensure that
 for each $B\in\mathcal B$ one has
\begin{equation}\label{eqe6}
\frac{\int_B D_j^{n-1}(\theta)\,d\theta}{\int_B D^{n-1}(\theta)\,d\theta}=\Big(1+{O}(\epsilon')\Big)\frac{\int_{\varrho_{\Omega_j}(B)} |\nabla u_{K_j}|^{n-1}\,dS}{\int_{\varrho_\Omega(B)} |\nabla u_K|^{n-1}\,dS},
\end{equation}
where $O(\epsilon')$ is a positive big-oh function of $\epsilon'$.

Next, we are about to show that $c'_B$ in (\ref{eqe5}) is equal to $1$. To this end, let us fix $s_0$ and allow $\eta$ to rely on $s_0$. Note that the quotient on the right side of (\ref{eqe6}) is the ratio of the $n$-harmonic measures (cf. \cite{Lin}) of the sets $\varrho_j(B)$ and $\varrho(B)$. So, employing the maximum principle to compare $n$-harmonic functions in $\mathbb R^n\setminus K_j$ to $n$-harmonic functions in $\mathbb R^n\setminus\rho K$ (where $\rho K$ means a $\rho$-dilation of $K$) , we can take $\eta>0$ smaller still, relying on $s_0$ such that
\begin{equation}\label{eqe7}
\left|\frac{\int_B D_j^{n-1}(\theta)\,d\theta}{\int_B D^{n-1}(\theta)\,d\theta}-1\right|\lesssim\epsilon'
\end{equation}
holds for any $B\in\mathcal B$. In the above and below, $U\lesssim V$ stands for $U\le c_nV$ for a dimensional constant $c_n>0$.

Using the $q>n$-harmonic setting of Lewis-Nystr\"om's \cite[Theorem 3]{LewN} and the H\"older inequality we find that
\begin{equation}\label{eqe8}
\left(\frac{1}{S(\varrho_\Omega(B))}\int_{\varrho_\Omega(B)}|\nabla u_K|^n\,dS\right)^\frac{n-1}{n}\lesssim\frac{1}{S(\varrho_K(B))}\int_{\varrho_K(B)}|\nabla u_K|^{n-1}\,dS
\end{equation}
is valid for any ball centered at $\partial K$. Clearly, a similar estimate is valid for each $\partial K_j$. Thus,
\begin{equation}\label{eqe9}
\left(s_0^{1-n}\int_{B}D^n(\theta)\,d\theta\right)^\frac{n-1}{n}\lesssim s_0^{1-n}\int_{B}D^{n-1}(\theta)\,d\theta
\end{equation}
 and similarly for $D_j$. Now, using H\"older's inequality plus (\ref{eqe9}), (\ref{eqe5}) and (\ref{eqe8}), we get that for each $B\in\mathcal B$,
 \begin{eqnarray*}
&& \frac{\int_B c'_B D_j^{n-1}(\theta)\,d\theta}{\int_B D^{n-1}(\theta)\,d\theta}-1\\
&&=\frac{\int_B c'_B \Big(\big(\frac{D_j(\theta)}{D(\theta)}\big)^{n-1}-1\Big)D^{n-1}(\theta)\,d\theta}{\int_B D^{n-1}(\theta)\,d\theta}\\
 &&\lesssim\left(\int_B \Big(c'_B\big(\frac{D_j(\theta)}{D(\theta)}\big)^{n-1}-1\Big)^n\,d\theta\right)^\frac1n
 \left(\frac{\Big(\int_B D^{n}(\theta)\,d\theta\Big)^\frac{n-1}{n}}{\int_B D^{n-1}(\theta)\,d\theta}\right)\\
 &&\lesssim\left(s_0^{1-n}\int_B \Big(c'_B\big(\frac{D_j(\theta)}{D(\theta)}\big)^{n-1}-1\Big)^n\,d\theta\right)^\frac1n\lesssim\epsilon'
 \end{eqnarray*}
 In a similar manner, we replace $c'_B D_j/D$ by $(D/c'_B)D_j$ in the above estimates to obtain
 $$
 \frac{\int_B D^{n-1}(\theta)\,d\theta}{\int_B c'_B D_j^{n-1}(\theta)\,d\theta}-1\lesssim\epsilon'.
 $$
 Since (\ref{eqe7}) yields
 $$
\left|\frac{\int_B D_j^{n-1}(\theta)\,d\theta}{\int_B D^{n-1}(\theta)\,d\theta}-1\right|\lesssim\epsilon',
 $$
 we must have $|c'_B-1|\lesssim\epsilon'$, whence getting $c'_B=1$. As a consequence of this and (\ref{eqe5}), we find
 $$
 s_0^{1-n}\int_B\Big|\Big(\frac{D_j(\theta)}{D(\theta)}\Big)^{n-1}-1\Big|^\alpha\,d\theta\lesssim\epsilon'\ \ \&\ \
 s_0^{1-n}\int_B\Big|\Big(\frac{D(\theta)}{D_j(\theta)}\Big)^{n-1}-1\Big|^\alpha\,d\theta\lesssim\epsilon',
$$
whence completing the proof of (iv).

Although the idea of verifying (v) is motivated by the argument for \cite[Proposition 4.3]{Je96a}, we still need more effort to adapt it to our nontrivial situation. Because of $q>n$ in \cite[Theorem 3]{LewN}, it is possible to find $\beta\in (1,\infty)$ such that
$n\beta/(\beta-1)=q$. Given $\epsilon>0$, take $\eta>0$ and $F$ in
accordance with (i)-(iv). Using the inequality
$$
|a^n-b^n|\le\frac{(a+b)|a^{n-1}-b^{n-1}|}{n^{-1}(n-1)}\quad\forall\
a,b\ge 0\,,
$$
the H\"older inequality and (\ref{e42a}), we achieve

\begin{eqnarray*}
&&\int_{F}|D_j^n(\theta)-D^n(\theta)|\,d\theta\\
&&\le\Big(\frac{n}{n-1}\Big)\int_{F}\big|D_j^{n-1}(\theta)-D^{n-1}(\theta)\big|\big(D_j(\theta)+D(\theta)\big)\,d\theta\\
&&\lesssim\left(\int_{F}\big|D_j^{n-1}(\theta)-D^{n-1}(\theta)\big|^\frac{n}{n-1}\,d\theta\right)^\frac{n-1}{n}\left(\int_F\big(D_j(\theta)+D(\theta)\big)^n
\,d\theta\right)^\frac{1}{n}\\
&&\lesssim\big(2\sigma_{n-1}\big)^\frac{1}{n}\left(\int_{F}\Big|\Big(\frac{D_j(\theta)}{D(\theta)}\Big)^{n-1}-1\Big|^\frac{n}{n-1}D^n(\theta)\,d
S(\theta)\right)^\frac{n-1}{n}\\
&&\lesssim\left(\int_{F}\Big|\Big(\frac{D_j(\theta)}{D(\theta)}\Big)^{n-1}-1\Big|^\frac{n\beta}{n-1}\,d
S(\theta)\right)^\frac{n-1}{n\beta}\left(\int_F D^q(\theta)\,d\theta\right)^\frac{n-1}{q},
\end{eqnarray*}
thereby deducing
\begin{equation}\label{eqeF}
\int_{F}\big|D_j^n(\theta)-D^n(\theta)\big|\,d\theta\lesssim\epsilon\quad\hbox{as}\quad j\to \infty,
\end{equation}
 through (iv) with $\alpha=q$ as well as \cite[Theorem 3]{LewN} insuring
 $ \int_{\mathbb S^{n-1}} D^q(\theta)\,d\theta<\infty.$

On the other hand, by the H\"older inequality we derive
\begin{eqnarray*}
&&\int_{\mathbb S^{n-1}\setminus F}|D_j^n(\theta)-D^n(\theta)|\,d\theta\\
&&\le\int_{\mathbb S^{n-1}\setminus F}\big(D_j^n(\theta)+D^n(\theta)\big)\,d\theta\\
&&\lesssim\big(S(\mathbb S^{n-1}\setminus
F)\big)^\frac{q}{q-n}\left(\int_{\mathbb S^{n-1}\setminus
F}\big(D_j^q(\theta)+D^q(\theta)\big)\,d\theta\right)^\frac{n}{q},
\end{eqnarray*}
whence getting (v) through (iii), (\ref{eqeF}) and \cite[Theorem 3]{LewN} which especially guarantees
$$
\sup_{j}\int_{\mathbb S^{n-1}\setminus
F}\big(D_j^q(\theta)+D^q(\theta)\big)\,d\theta<\infty.
$$
\end{proof}

With the help of Theorem \ref{Le}, we can establish the following weak convergence result for the measure induced by Theorem \ref{p-41}.

\begin{theorem}\label{th51} Let $K, K_j\in\mathbb K^n$ and $\lim_{j\to\infty}d_H(K_j,K)=0$. If $u,u_j$ are the $n$-equilibrium potentials of $K,K_j$ respectively, then $d\mu_j=(g_j)_\ast(|\nabla u_j|^n\,dS)$ converges weakly to $d\mu=g_\ast(|\nabla u|^n\,dS)$, i.e.,
$$
\lim_{j\to\infty}\int_{\mathbb S^{n-1}}f\,d\mu_j=\int_{\mathbb S^{n-1}}f\,d\mu\quad\forall\quad f\in C(\mathbb S^{n-1}).
$$
\end{theorem}
\begin{proof} The following argument is analogous to \cite[Section 5]{CheY}. 
Recall that the push-forward measures $d\mu\ \&\ d\mu_j$ on $\mathbb S^{n-1}$ are determined respectively by
$$
\mu(E)=\int_{g^{-1}(E)}|\nabla u|^n\,dS\ \ \& \ \
\mu_j(E)=\int_{g_j^{-1}(E)}|\nabla u_j|^n\,dS\ \forall\ \ \hbox{Borel\ set}\ \ E\subseteq\mathbb S^{n-1},
$$
where $g$ and $g_j$ are the Gauss maps attached to $K$ and $K_j$ respectively. It remains to verify that $\mu$ is the weak limit of $\mu_j$ as $j\to\infty$.

An application of Theorem \ref{Le}(v) yields
\begin{equation}\label{eqeD}
\lim_{j\to\infty}\Big(\mu(\mathbb S^{n-1})-\mu_j(\mathbb S^{n-1})\Big)=\lim_{j\to\infty}\int_{\mathbb S^{n-1}}\big(D^n(\theta)-D^n_j(\theta)\big)\,d\theta=0.
\end{equation}
Note that ${g}^{-1}(E)\subseteq\partial K$ and ${g}_j^{-1}(E)\subseteq\partial K_j$
are closed (cf. \cite{CheY} and \cite{Je91, Je96a}) for any Borel set $E\subseteq\mathbb S^{n-1}$, and that if $\xi_j\in {g}_j(x_j)$ approaches $\xi$ and if $x_j\to x$ then $\xi\in {g}(x)$ and $x\in\partial K$. So, for any open neighborhood $U$ in $\partial K$ of the closed set ${g}^{-1}(E)$ we have that $\varrho_{K_j}^{-1}\big({g}_j^{-1}(E)\big)\subseteq \varrho^{-1}_{K}(U)$ as $j\to\infty$, whence finding that
\begin{equation*}
\limsup_{j\to\infty}\mu_j(E)\le\lim_{j\to\infty}\int_{\varrho_{K}^{-1}(U)}D_j^n(\theta)\,d\theta\le
\int_{\varrho_{K}^{-1}(U)}D^n(\theta)\,d\theta.
\end{equation*}
When the infimum ranges over all $U\supseteq {g}^{-1}(E)$, we get
$\limsup_{j\to\infty}\mu_j(E)\le\mu(E).$
This last inequality and (\ref{eqeD}) imply that for any open subset $O$ of $\mathbb S^{n-1}$,
\begin{eqnarray*}\label{eqeLI}
\liminf_{j\to\infty}\mu_j(O)&=&\liminf_{j\to\infty}\big(\mu_j(O)-\mu_j(\mathbb S^{n-1}\setminus O)\big)\\
&\ge&\liminf_{j\to\infty}\mu_j(\mathbb S^{n-1})-\mu(\mathbb S^{n-1}\setminus O)\\
&=& \mu(\mathbb S^{n-1})-\mu(\mathbb S^{n-1}\setminus O)=\mu(O).
\end{eqnarray*}
If $\tilde{\mu}$ is any weak limit of a subsequence of $\mu_j$, then the above inequalities on $\limsup_{j\to\infty}$ and $\liminf_{j\to\infty}$ deduce that
$\tilde{\mu}(C)\le{\mu}(C)$ and $\mu(O)\le\tilde{\mu}(O)$
hold for any closed $C\subseteq\mathbb S^{n-1}$ and any open $O\subseteq\mathbb S^{n-1}$. Consequently, for any closed $C\subseteq\mathbb S^{n-1}$ one has
$$
\mu(C)\ge\tilde{\mu}(C)=\inf\{\tilde{\mu}(O): \hbox{open}\ O\supset C\}\ge\inf\{\mu(O):\ \hbox{open}\ O\supset C\}={\mu}(C),
$$
and hence $\tilde{\mu}=\mu$.
\end{proof}

The following is the general variational result.

\begin{theorem}\label{t41g} (\ref{e42b})-(\ref{e43a})-(\ref{e43b}) are valid for $K_0,K_1\in\mathbb K^n$.
\end{theorem}
\begin{proof} Given $K_0,K_1\in\mathbb K^n$. There are two sequences of $C^2$ strictly convex bodies $K_{0,j}, K_{1,j}$ such that
$$
\lim_{j\to\infty}d_H(K_{0,j},K_0)=0=\lim_{j\to\infty}d_H(K_{1,j},K_1).
$$
Now, for $t\in (0,1)$ and $j=1,2,...$ set
$$
\begin{cases}
& K_t=(1-t)K_0+tK_1,\ \ K_{t,j}=1-t)K_{0,j}+tK_{1,j};\\
& \Phi(t)=\hbox{ncap}(K_0+tK_1),\ \ \Phi_j(t)=\hbox{ncap}(K_{0,j}+tK_{1,j});\\
& \Psi(t)=\hbox{ncap}(K_t),\ \ \Psi_j(t)=\hbox{ncap}(K_{t,j}).
\end{cases}
$$
Note that
$$
t\mapsto\Psi_j(t)=(1-t)\Phi_j\Big(\frac{t}{1-t}\Big)
$$
is a concave function on $(0,1)$. So,
\begin{equation}\label{eIn}
\Psi_j'(t)\le\frac{\Psi_j(t)-\Psi_j(0)}{t}\le\Psi_j'(0)\quad\forall\quad t\in (0,1).
\end{equation}
A simple computation gives
$$
\Psi_j'(t)=-\Phi_j\Big(\frac{t}{1-t}\Big)+(1-t)^{-1}\Phi_j'\Big(\frac{t}{1-t}\Big)
$$
and
\begin{eqnarray*}
\Psi_j'(0)&=&-\Phi_j(0)+\Phi_j'(0)\\
&=&\frac{\hbox{ncap}(K_{0,j})}{\sigma_{n-1}}\left(-\sigma_{n-1}+\int_{\partial K_{0,j}}h_{K_{1,j}}(g)|\nabla u_{K_0,j}|^n\,dS\right)\\
&=&\frac{\hbox{ncap}(K_{0,j})}{\sigma_{n-1}}\int_{\partial K_{0,j}}\Big(h_{K_{1,j}}(g)-h_{K_{0,j}}(g)\Big)|\nabla u_{K_0,j}|^n\,dS,
\end{eqnarray*}
owing to (\ref{e42a}) and (\ref{e43a}). Upon letting $j\to\infty$ and $t\to 0$ in (\ref{eIn}), we use Theorem \ref{th51} to obtain
$$
\Psi'(0)=\frac{\hbox{ncap}(K_{0})}{\sigma_{n-1}}\int_{\partial K_{0}}
\Big(h_{K_{1}}(g)-h_{K_{0}}(g)\Big)|\nabla u_{K_0}|^n\,dS,
$$
whence establishing (\ref{e43a}), equivalently, (\ref{e42b}), and thus (\ref{e43b}).
\end{proof}

\section{Minkowski's problem for conformal capacities}\label{s6}

\setcounter{equation}{0}

\subsection{Prescribing volume variation}\label{s61} Given $K\in\mathbb K^n$. From the Gauss map $g: \partial K\to \mathbb S^{n-1}$ one can introduce the area function $\mathcal{H}^{n-1}_{\partial K}$ of $\partial K$ via setting
$$
\mathcal{H}^{n-1}_{\partial K}(E)=S\big(\{x\in\partial K:\ g(x)\cap E\not=\emptyset\})\quad\forall\quad\hbox{Borel\ subset}\ E\subseteq \mathbb S^{n-1}.
$$
This measure $d\mathcal{H}^{n-1}_{\partial K}$ is treated as the push-forward measure $g_\ast(dS)$ on $\mathbb S^{n-1}$ of the $n-1$ dimensional surface measure $dS$ on $\partial K$ through the inverse map $g^{-1}$ of $g$. Obviously, $\mathcal{H}^{n-1}_{\partial K}(\mathbb S^{n-1})=S(K)$, i.e., the surface area of $K$. Two more special facts on this measure are worth recalling. The first is that if $\partial K$ is polyhedron then  $d\mathcal{H}^{n-1}_{\partial K}=\sum_k c_k \delta_{\nu_k}$, where $\delta_{\nu_k}$ is the unit point mass at $\nu_k$ and $c_k$ is the $(n-1)$ dimensional measure of the face of $\partial K$ with outward unit normal being $\nu_k$. The second is that if $\partial K$ is strictly convex and smooth then $d\mathcal{H}^{n-1}_{\partial K}$ is absolutely continuous and so decided by $1/G(K,\cdot)$, where $G(K,\cdot)$ is the Gauss curvature of $\partial K$.

The classical Minkowski problem is to ask under what conditions on a given nonnegative Borel measure on $\mathbb S^{n-1}$ one can get a convex body $K\in\mathbb K^n$ such that
$d\mathcal{H}^{n-1}_{\partial K}=d\mu$. As well known, this problem is solvable if and only if the support of $\mu$ is not contained in any closed hemisphere and
$$
\int_{\mathbb S^{n-1}}\theta\cdot\xi\,d\mu(\xi)=0\quad\forall\quad\theta\in\mathbb S^{n-1}.
$$
Moreover, the above $K$ is unique up to translation -- this follows from the equality case of the well-known Brunn-Minkowski inequality for $V(\cdot)$:
$$
V(K_0+tK_1)^\frac1n\ge V(K_0)^\frac1n+tV(K_1)^\frac1n\quad\forall\quad K_0, K_1\in\mathbb K^n\quad\&\quad t\in [0,1].
$$
The foregoing inequality and the following Hadamard's variation formula:
$$
\frac{d}{dt}V(K_0+tK_1)\big|_{t=0}=\int_{\partial K_0}h_{K_1}(g)\,dS=\int_{\mathbb S^{n-1}}h_{K_1}\,d\mathcal{H}_{\partial K_0}^{n-1}\quad\forall\quad K_0, K_1\in\mathbb K^n
$$
give
$$
\int_{\mathbb S^{n-1}}h_{K_1}\,d\mathcal{H}^{n-1}_{\partial K_0}\ge nV(K_0)^{1-\frac{1}{n}}V(K_1)^\frac1n,
$$
whence ensuring that if $K_0$ is fixed and $K_1$ varies with $V(K_1)\ge 1$ then
$\int_{\mathbb S^{n-1}}h_{K_1}\,d\mathcal{H}^{n-1}_{\partial K_0}$
 reaches its minimum whenever $K_1=V(K_0)^{-\frac1n}K_0$. So, the just-described Minkowski problem is equivalent to the problem prescribing the first variation of volume, i.e., the following minimizing problem
 $$
 \inf\left\{\int_{\mathbb S^{n-1}} h_K\,d\mu:\ K\in\mathbb K^n\ \&\ V(K)\ge 1\right\}
 $$
 for a given nonnegative Borel measure $\mu$ on $\mathbb S^{n-1}$; see e.g.\cite{Col, Pog, Nir}.

\subsection{Prescribing conformal capacitary variation}\label{s62} As $V(\cdot)$ is replaced by $\hbox{ncap}(\cdot)$, we empoy Theorem \ref{t41} and (\ref{e41}) to obtain that
\begin{equation*}\label{e41add}
\int_{\partial K_0}h_{K_1}(g)|\nabla u_{K_0}|^n\,dS=\Big(\frac{\sigma_{n-1}}{\hbox{ncap}(K_0)}\Big)\frac{d}{dt}\hbox{ncap}(K_0+tK_1)\Big|_{t=0}
\ge\frac{\sigma_{n-1}\hbox{ncap}(K_1)}{\hbox{ncap}(K_0)}
\end{equation*}
holds for all $K_0,K_1\in\mathbb K^n$. Clearly, if $K_0\in\mathbb K^n$ is fixed and $K_1\in \mathbb K^n$ changes under $\hbox{ncap}(K_1)\ge 1$, then
$$
\int_{\mathbb S^{n-1}}h_{K_1}\,g_\ast(|\nabla u_{K_0}|^n\,dS)=\int_{\partial K_0}h_{K_1}(g)|\nabla u_{K_0}|^n\,dS\ge \frac{\sigma_{n-1}}{\hbox{ncap}(K_0)}
$$
with equality (i.e., the most right quantity exists as the infimum of the most left integral) if $K_1=K_0/\hbox{ncap}(K_0)$. This implication plus the review about the problem of prescribing the first variation of volume leads to a consideration of the Minkowski type problem for the first variation of conformal capacity. Below is our result.

\begin{theorem}\label{t53} Let $\mu$ be a nonnegative Borel measure on $\mathbb S^{n-1}$ and
$$
\mathcal{P}_\mu(\xi)=\int_{\mathbb S^{n-1}}\max\{0,\xi\cdot\eta\}\,d\mu(\eta).
$$
If
\begin{equation}\label{e48}
0<\min_{\xi\in\mathbb S^{n-1}}\mathcal{P}_\mu(\xi)\le\max_{\xi\in\mathbb S^{n-1}}\mathcal{P}_\mu(\xi)<\infty
\end{equation}
and
\begin{equation}\label{e49}
\mathcal{P}_\mu(\xi)=\mathcal{P}_\mu(-\xi)\quad\forall\quad\xi\in\mathbb S^{n-1},
\end{equation}
then
\begin{equation}\label{e410}
\mathcal{M}_{\hbox{ncap}}=\inf\left\{\int_{\mathbb S^{n-1}}h_Kd\mu:\ K\in\mathcal C^n\ \&\ \hbox{ncap}(K)\ge 1\right\}>0,
\end{equation}
and there is a $K\in\mathcal{C}^n$ such that
\begin{equation}\label{min}
\mathcal{M}_{\hbox{ncap}}=\int_{\mathbb S^{n-1}}h_K\,d\mu\ \ \&\ \ \hbox{ncap}(K)\ge 1.
\end{equation}
Moreover, if (\ref{e410}) has a minimizer $K\in\mathbb K^n$ with
\begin{equation}\label{e411}
g_\ast(|\nabla u_K|^n\,dS)=\mu\ \&\ \ \hbox{ncap}(K)=1,
\end{equation}
then such a $K$ is unique up to translation.
\end{theorem}

\begin{proof} To prove $\mathcal{M}_{\hbox{cap}}>0$, observe that (\ref{e49}) ensures that
$\int_{\mathbb S^{n-1}}h_K d\mu$ is translation invariant. So, we may assume that the origin is at the midpoint of a diameter of $K\in\mathcal C^n$ with $\hbox{ncap}(K)\ge 1$. Let $2R=\hbox{diam}(K)$. According to Theorem \ref{t32}, we have:
$$
\hbox{ncap}(K)\ge 1\Rightarrow 2R\ge 2\hbox{ncap}(K)\ge 2.
$$
If $\mathbf{e}$ is a unit vector with $\pm R\mathbf{e}\in\partial K$, then $h_K(\xi)\ge R|\mathbf{e}\cdot\xi|$ holds for all $\xi\in\mathbb S^{n-1}$, and hence
$$
2\min_{\mathbb S^{n-1}}\mathcal{P}_\mu\le 2R\mathcal{P}_\mu(\mathbf{e})\le\int_{\mathbb S^{n-1}}R|\mathbf{e}\cdot\xi|\,d\mu(\xi)\le\int_{\mathbb S^{n-1}}h_K\,d\mu.
$$
Using (\ref{e48}), we get $\mathcal{M}_{\hbox{ncap}}>0$.

Furthermore, when $K\in\mathcal C^n$ satisfies
$$
\hbox{ncap}(K)\ge 1\quad\&\quad\int_{\mathbb S^{n-1}}h_K\,d\mu\le 2\mathcal{M}_{\hbox{ncap}},
$$
an upper bound of the diameter of $K$ can be determined through
$$
\hbox{diam}(K)\min_{\mathbb S^{n-1}}\mathcal{P}_\mu=2R\min_{\mathbb S^{n-1}}\mathcal{P}_\mu\le 2\mathcal{M}_{\hbox{ncap}}.
$$
Suppose $\{K_j\}_{j=1}^\infty$ is a sequence in $\mathcal C^n$ that satisfies
$$
\mathcal{M}_{\hbox{ncap}}=\lim_{j\to\infty}\int_{\mathbb S^n}h_{K_j}\,d\mu\ \ \&\ \ \hbox{ncap}(K_j)\ge 1.
$$
Then
$$
2\le 2\hbox{ncap}(K_j)\le \hbox{diam}(K_j)\le \frac{2\mathcal{M}_{\hbox{ncap}}}{\min_{\mathbb S^{n-1}}\mathcal{P}_\mu}\ \ \hbox{as}\ \ j\to\infty.
$$
In accordance with the Blaschke selection principle (see e.g. \cite[Theorem 1.8.6]{Schn}), $\{K_j\}_{j=1}^\infty$ has a subsequence, still denoted by $\{K_j\}_{j=1}^\infty$, that converges to a $K\in\mathcal C^n$ with respect to the Hausdorff distance $d_H(\cdot,\cdot)$. Consequently, $h_{K_j}\to h_K$. Now, the continuity of $\hbox{ncap}(\cdot)$ ensures $\hbox{ncap}(K)\ge 1$ and so (\ref{min}) holds.

Our argument for the uniqueness is inspirited by \cite{CJeL}. Assume now that $K_0, K_1\in\mathbb K^n$ are two minimizers of (\ref{e410}) and satisfy (\ref{e411}). Then
$$
\begin{cases}
g_\ast(|\nabla u_{K_0}|^n\,dS)=g_\ast(|\nabla u_{K_1}|^n\,dS);\\
\hbox{ncap}(K_0)=1=\hbox{ncap}(K_1).
\end{cases}
$$
If $\psi(t)=\hbox{ncap}\big((1-t)K_0+tK_1\big)$, then Theorems \ref{t41g} \& \ref{p-41} yield
\begin{eqnarray*}
\psi'(0)&=&\frac{\hbox{ncap}(K_0)}{\sigma_{n-1}}\int_{\partial K_0}\big(h_{K_1}(g)-h_{K_0}(g)\big)|\nabla u_{K_0}|^n\,dS\\
&=&\sigma_{n-1}^{-1}\Big(\int_{\partial K_0}h_{K_1}(g)|\nabla u_{K_0}|^n\,dS-\sigma_{n-1}\Big)\\
&=&{\sigma_{n-1}}^{-1}\Big(\int_{\mathbb S^{n-1}}h_{K_1}\,g_\ast(|\nabla u_{K_0}|^n\,dS)-\sigma_{n-1}\Big)\\
&=&{\sigma_{n-1}}^{-1}\Big(\int_{\mathbb S^{n-1}}h_{K_1}\,g_\ast(|\nabla u_{K_1}|^n\,dS)-\sigma_{n-1}\Big)\\
&=&{\sigma_{n-1}}^{-1}\Big(\sigma_{n-1}-\sigma_{n-1}\Big)=0.
\end{eqnarray*}
Note that $t\mapsto\psi(t)$ is concave on $[0,1]$. So this function is constant, in particular, one has
\begin{equation}\label{ef}
\hbox{ncap}(K_1)=\psi(1)=\psi(t)=\psi(0)=\hbox{ncap}(K_0).
\end{equation}
 Since the equality of (\ref{e41}) holds, $K_1$ is a translate and a dilate of $K_0$. But  (\ref{ef}) is valid, so $K_1$ is only a translate of $K_0$ thanks to \cite{ColC}.
\end{proof}

\section{Yau's problem for conformal capacities}\label{s7}
\setcounter{equation}{0}

\subsection{Prescribed mean curvature problem}\label{s71} On page 683 of \cite{Yau}, S.-T. Yau posed the following problem: 

``{Let $h$ be a real-valued function on $\mathbb R^3$. Find (reasonable) conditions on $h$ to insure that one can find a closed surface with prescribed genus in $\mathbb R^3$ whose mean curvature (or curvature) is given by $h$. F. Almgren made the following comments: For ``suitable" $h$ one can obtain a compact smooth submanifold $\partial A$ in $\mathbb R^3$ having mean curvature $h$ by maximizing over bonded open sets $A\subset\mathbb R^3$ the quantity
$$
F(A)=\int_A h\,d\mathcal L^3-Area(\partial A).
$$
A function $h$ would be suitable, for example, in case it were continuous, bounded, and $\mathcal L^3$ summable, and $\sup F>0$. However, the relation between $h$ and the genus of the resulting extreme $\partial A$ is not clear.}"

Although not yet completely solved, this problem for mean curvature or Gaussian curvature has a solution at least for the closed surface of genus zero, see \cite{TrW, BaK, HSW} or \cite{Tso1, Tso2}. The following, essentially contained in \cite[Corollary 1.2]{X}, may be regarded as a resolution of Yau's problem in a weak sense - if $I\in L^1(\mathbb R^n)$ is positive and continuous, $k$ is nonnegative integer, $\alpha\in (0,1)$, $S(\cdot)$ stands for the surface area, and 
$$
\mathcal{I}(K)=S(K)-\int_K I\,dV,
$$
then one has:

\begin{itemize}
\item $\mathcal{I}(\cdot)$ attains its infimum over $\mathcal{C}^n$ when and only when there is $K\in\mathcal{C}^n$ such that $\mathcal{I}(K)\le 0$.

\item Suppose $K\in\mathbb K^n$ is a minimizer for $\mathcal{I}(\cdot)$. Then there is a curvature measure $\mu_K$ on $\mathbb S^{n-1}$ such that the so-called weak mean curvature equation
\begin{equation}\label{mean}
\int_{\mathbb S^{n-1}}\phi\,d\mu_K=\int_ {\mathbb S^{n-1}}\phi g_\ast(I\,dS)\quad\forall\ \ \phi\in C(\mathbb S^{n-1})
\end{equation}
holds. Moreover, if $\partial K$ is $C^2$ strictly convex then the classic mean curvature equation $H(K,x)=I(x)$ is valid for all $x\in\partial K$.

\item If $I$ is of $C^{k,\alpha}(\mathbb R^n)$ and $K\in\mathbb K^n$, with $C^2$ strictly convex boundary $\partial K$, is a minimizer for $\mathcal{I}(\cdot)$, then $\partial K$ is of $C^{k+2,\alpha}$.
\end{itemize}

\subsection{Prescribing conformal capacitary curvature}\label{s72} Thanks to the relationship between the surface area and the conformal capacity explored in Section \ref{s33}, as well as the discussion on the Minkowski type problem above, it seems interesting to consider the conformal capacity analogue of Yau's problem. More precisely, using the conformal capacity in place of the surface area we study the functional
$$
\mathcal{J}(K)=\hbox{ncap}(K)-\int_K J\,dV,
$$
thereby obtaining the following result.

\begin{theorem}\label{t71} Let $J\in L^1(\mathbb R^n)$ be positive and continuous, $k$ nonnegative integer, and $\alpha\in (0,1)$.

\item{\rm(i)} $\mathcal{J}(\cdot)$ attains its infimum over $\mathcal{C}^n$ when and only when there exists $K\in\mathcal{C}^n$ such that $\mathcal{J}(K)\le 0$.

\item{\rm(ii)} Suppose $K\in\mathbb K^n$ is a minimizer for $\mathcal{J}(\cdot)$. Then such a $K$ satisfies the so-called weak conformal capacitary curvature equation
\begin{equation}\label{e71}
\Big(\frac{\hbox{ncap}(K)}{\sigma_{n-1}}\Big)\int_{\mathbb S^{n-1}}\phi g_\ast\big(|\nabla u_K|^n\,dS)=\int_{\mathbb S^{n-1}}\phi g_\ast(J\,dS)\ \ \forall\ \phi\in C(\mathbb S^{n-1}).
\end{equation}
Moreover, if $\partial K$ is $C^2$ strictly convex then the so-called conformal capacitary curvature equation $\hbox{ncap}(K)\sigma_{n-1}^{-1}|\nabla u_K(x)|^n=J(x)$ is valid for all $x\in\partial K$.

\item{\rm(iii)} If $J$ is of $C^{k,\alpha}(\mathbb R^n)$ and $K$, with $\partial K$ being $C^2$ strictly convex, is a minimizer for $\mathcal{J}(\cdot)$, then $\partial K$ is of $C^{k+1,\alpha}$.
\end{theorem}

\begin{proof} (i) Due to $J\in L^1(\mathbb R^n)$, we have
$$
\mathcal{J}(K)\ge\hbox{ncap}(K)-\|J\|_{L^1(\mathbb R^n)}\quad\forall\quad K\in\mathcal{C}^n.
$$
Note that if a sequence of balls $\{B_j\}$ converges to a point then $\{\mathcal{J}(B_j)\}$ goes to $0$. So, $\inf_{K\in\mathcal{C}^n}\mathcal{J}(K)\le 0$. Consequently, if $\mathcal{J}(\cdot)$ attains its infimum at $K_0\in \mathcal{C}^n$ then there must be $\mathcal{J}(K_0)=\inf_{K\in\mathcal{C}^n}\mathcal{J}(K)\le 0$. Conversely, suppose there is $K\in\mathcal{C}^n$ such that $\mathcal{J}(K)\le 0$. Then $\inf_{K\in\mathbb K^n}\mathcal{J}(K)\le 0$. If $\{K_j\}$ is a sequence of minimizers for $\mathcal{J}(\cdot)$ with $\mathcal{J}(K_j)<0$ and the inradius of $K_j$ having a uniform lower bound $r_0>0$ (if, otherwise, $K_j$ tends to a set of single point $\{x_0\}$, then $\mathcal{J}(K_j)\to 0$ and hence $\{x_0\}\in \mathcal{C}^n$ is a minimizer). Using this and (\ref{e34}) we get
\begin{equation}\label{e77}
2r_0\le 2\hbox{ncap}(K_j)\le \hbox{diam}(K_j).
\end{equation}
Since
$$
\mathcal{J}(K_j)\ge \Big(\frac{V(K_j)}{\omega_n}\Big)^{\frac{1}{n}}-\|J\|_{L^1(\mathbb R^n)},
$$
if $\hbox{diam}(K_j)$ is unbounded, then (\ref{e77}) is used to imply that $V(K_j)$ is unbounded, and hence $\mathcal{J}(K_j)$ is unbounded from above. But, $\mathcal{J}(K_j)<0$. Therefore, $\hbox{diam}(K_j)$ has a uniform upper bound. Now, taking into account of the above-mentioned Blaschke selection principle, we may get a subsequence of $\{K_j\}$ which is convergent to an element $K_0\in\mathbb K^n$. Clearly, $\mathcal{J}(\cdot)$ is continuous. Thus, $K_0$ is a minimizer of $\mathcal{J}(\cdot)$.

(ii) For $K\in\mathbb K^n$, $t>0$ and $\phi\in C^1(\mathbb S^{n-1})$ let
$$
K_t=\big\{x\in\mathbb R^n:\ x\cdot \theta\le h_K(\theta)+t\phi(\theta)\quad\forall\quad \theta\in\mathbb S^{n-1}\big\}.
$$
Then $K_t\in\mathbb K^n$ and $h_{K_t}=h_K+t\phi$. Using Theorem \ref{t41g} (plus the ideas presented in \cite[Sections 3-4]{Je96b}) as well as Tso's variation formula \cite[(4)]{Tso2} once again, we produce
\begin{equation}\label{e72}
\frac{d}{dt}\mathcal{J}(K_t)\Big|_{t=0}=\Big(\frac{\hbox{ncap}(K)}{\sigma_{n-1}}\Big)\int_{\partial K}\phi(g)|\nabla u_K|^n\,dS-\int_{\partial K}\phi(g) J\, dS.
\end{equation}
Obviously, if $K$ is a minimizer of $\mathcal{J}(\cdot)$, then it is a critical point of $\mathcal{J}(K_t)$ and hence
$\frac{d}{dt}\mathcal{J}(K_t)\Big|_{t=0}=0.$ This and (\ref{e72}) give
\begin{eqnarray*}
\Big(\frac{\hbox{ncap}(K)}{\sigma_{n-1}}\Big)\int_{\mathbb S^{n-1}}\phi g_\ast(|\nabla u_K|^n\,dS)&=&\Big(\frac{\hbox{ncap}(K)}{\sigma_{n-1}}\Big)\int_{\partial K} \phi(g)|\nabla u_K|^n\,dS\\
&=&\int_{\partial K}\phi(g) J\, dS\\
&=&\int_{\mathbb S^{n-1}}\phi g_\ast(J\,dS).
\end{eqnarray*}
Owing to the fact that $\phi\in C^1(\mathbb S^{n-1})$ is arbitrary, we arrive at (\ref{e71}). Furthermore, if $\partial K$ is $C^2$ strictly convex, then $g:\partial K\mapsto \mathbb S^{n-1}$ is a diffeomorphism (cf. \cite{CFG, Giu}), and hence one has
\begin{equation}
\label{e8e}
\Big(\frac{\hbox{ncap}(K)}{\sigma_{n-1}}\Big)|\nabla u_K(x)|^n=J(x)\quad\forall\quad x\in\partial K.
\end{equation}

(iii) Suppose $J\in C^{k,\alpha}(\mathbb R^n)$ with $k$ being a nonnegative integer. Since $\partial K$ is of $C^2$, an application of \cite[Theorem 1]{Lie} and \cite[Theorem 4.1]{MouY} (cf. \cite{GarS, DiB, Tol, UK, GiT}) yields that $u_K\in C^{1,\hat{\alpha}}(K)$ holds for some $\hat{\alpha}\in (0,1)$, and more importantly, the Gauss map from $\partial K$ to $\mathbb S^{n-1}$ is a diffeomorphism. Therefore,  (\ref{e8e}) is true. Using (\ref{e8e}) and $J\in C^{k,\alpha}(\mathbb R^n)$ with $\alpha\in (0,1)$, we obtain that $|\nabla u_K|\big|_{\partial K}$ is of $C^{k,\alpha}$. Note again that $\partial K$ is $C^2$ strictly convex. So, it follows that $\partial K$ is of $C^{k+1,\alpha}$ from the fact that $|\nabla u_K|\big|_{\partial K}$ is bounded above and below by two positive constants (cf. Theorem \ref{l5c}).
\end{proof}


\begin{thebibliography}{99}

\bibitem{AdH} D. R. Adams and L. I. Hedberg, {\it Function Spaces and Potential Theory}, {Springer-Verlag}, Berlin Heidelberg, 1996.

\bibitem{AikE} H. Aikawwa and M. Ess\'en, {\it Potential Theory - Selected Topics}, Lecture Notes in Math. 1633, 1996. DOI:10.1007/BFb0093410.

\bibitem{AndV} G. D. Anderson and M. K. Vamanamurthy, {\it The transfinite moduli of condensers in space}, Tohoko Math. J. {40}(1988)1-25.

\bibitem{AndVV} G. D. Anderson, M. K. Vamanamurthy and M. Vuorinen, {\it Conformal Invariants, Inequalities, and Quasiconformal Maps}, Wiley, 1997.

\bibitem{BaK} I. Ja. Bakel'man and B. E. Kantor,
{\it Estimates of the solutions of quasilinear elliptic equations that are connected with problems of geometry ``in the large''}, Mat. Sb. (N.S.) 91(133)(1973)336–349,471.


\bibitem{Bet} D. Betsakos, {\it Geometric versions of Schwarz's lemma for quasiregular mappings}, Proc. Amer. Math. Soc. 139(2010)1397-1407.

\bibitem{Bor} C. Borell, {\it Hitting probability of killed Brownian motion: A study on geometric regularity}, Ann. Sci. Ecole Norm. Sup\'er. Paris 17(1984)451-467.

\bibitem{Br} H. L. Bray, {\it Proof of the Riemannian Penrose inequality using the positive mass theorem}, J. Differential Geom. 59(2001)177-267.

\bibitem{CJeL} L. Caffarelli, D. Jerison and E. H. Lieb, {\it On the
    case of equality in the Brunn-Minkowski inequality for capacity},
  Adv. Math. {\bf 117}(1996)193-207.

\bibitem{CheY} S.-Y. Cheng and S.-T. Yau, {\it On the regularity of the
     solution of the $n$-dimensional Minkowski problem}, Comm. Pure
   Appl. Math. 29(1976)495-516.

\bibitem{Col} A. Colesanti, {\it Brunn-Minkowski inequalities for variational functionals and related problems}, Adv. Math. {194}(2005)105-140.

\bibitem{ColC} A. Colesanti and P. Cuoghi, {\it The Brunn-Minkowski inequality for the $n$-dimensional loarithmic capacity}, Potential Anal. 22(2005)289-304.

\bibitem{CLNSXYZ} A. Colesanti, E. Lutwak, K. Nystr\"om, P. Salani, J. Xiao, D. Yang and G. Zhang, {\it A Hadamard variational formula and a Minkowski problem for $p$-capacity}, preprint.

\bibitem{CFG} G. Crasta, I. Fragal\'a and F. Gazzola, {\it On a long-standing conjecture by P\'olya-Szeg\"o and related topics}, Z. angew. Math. Phys. 56(2005)763-782.

\bibitem{Dah} B. E. J. Dahlberg, {\it Estimates for harmonic measure}, Arch. Rational Mech. Anal. 65(1977)275-283.

\bibitem{DiB} E. DiBenendetto, {\it $C^{1,\alpha}$ local regularity of weak solutions of degenerate elliptic equations}, Nonlinear Anal. 7(1983)827-850.

\bibitem{EncP} A. Enciso and D. Peralta-Salas, {\it Symmetry for an overdetermined boundary problem in a punctured domain}, Nonlinear Anal. 70(2009)1080-1086.

\bibitem{EvaG} L. C. Evans and R. F. Gariepy, {\it Measure Theory and Fine Properties of Functions}, CRC Press LLC, 1992.

\bibitem{Flu} M. Flucher, {\it Variational Problems with Concentration}, PNLDE {36}, Birkh\"auser, 1999.

\bibitem{FreS} A. Freire and F. Schwartz, {\it Mass-capacity inequalities for conformally flat manifolds with boundary}, arXiv:11-7.1407v2 [math.DG] 20Jul2011.

\bibitem{Fug1} B. Fuglede, {\it The logarithmic potential in higher dimensions}, Mat. Fys. Medd. Dan. Vid. Selsk.  33:1(1960)1-14.

\bibitem{Fug2} B. Fuglede, {\it On generalized potentials of functions in the Lebesgue classes}, Math. Scand. 8(1960)287-304.

\bibitem{GarS} N. Garofalo and E. Sartori, {\it Symmetry in exterior boundary value problems for quasilinear elliptic equations via blow-up and a priori estimates}, Adv. Differential Equations 4(1999)137-161.

\bibitem{Ge0} F. W. Gehring, {\it Symmetrization of rings in space}, Trans. Amer. Math. Soc. 101(1961)499-519.

\bibitem{Ge1} F. W. Gehring, {\it Rings and quasiconformal mappings in space}, Trans. Amer. Math. Soc. 103(1962)353-393.

\bibitem{Ge2} F. W. Gehring, {\it Inequalities for condensers, hyperbolic capacity, and extremal lengths}, Michigan Math. J. 18(1971)1-20.

\bibitem{GiT} D. Gilbarg and N. Trudinger, {\it Elliptic Partial Differential Equations of Second Order}, 2nd edition, Springer-Verlag, 1983.

\bibitem{Giu} E. Giusti, {\it Minimal Surfaces and Functions of Bounded Variation}, Birkh\"auser, Boston, 1984.

\bibitem{Gri} P. Gritzmann, {\it A characterization of all loglinear inequalities for three quermassintegrals of convex bodies}, Proc. Amer. Math. Soc. 104(1988)563-570.

\bibitem{Gut} C. E. Guti\'errez, {The Monge-Amp\`ere Equation}, Progress in Nonlinear Differential Equations and Their Applications, Vol. 44, Birkh\"auser, 2001.

\bibitem{HKM} J. Heinonen, T. Kilpel\"ainen and O. Martio, {\it Nonlinear Potential Theory of Degenerate Elliptic Equations}, Dover Publications, Inc., Mineola, New York, 2006.

\bibitem{HS} A. Henrot and H. Shahgholian, {\it Existence of classical solutions to a free boundary problem for the $p$-Laplace operator: (I) the exterior convex case}, J. reine angew. Math. 521(2000)85-97.

\bibitem{Hil} E. Hille, {\it Analytic Function Theory}, Volume II, Ginn and Company, 1962.

\bibitem{HSW} Y.-J. Hsu, S.-J. Shiau and T.-H. Wang, {\it Graphs with prescribed mean curvature in the sphere,} Bull. Inst. Math. Academia Sin. 28:4(2000)215-223.

\bibitem{HI1} G. Huisken and T. Ilmanen, {\it The inverse mean curvature flow and the Riemannian Penrose inequality,} J. Differential Geom. 59(2001)353-437.


\bibitem{HuangW} L.-H. Huang and D. Wu, {\it Hypersurfaces with nonnegative scalar curvature}, arXiv:1102.5749v2[math.DG]19Jul2011.

\bibitem{Je91} D. Jerison, {\it Prescribing harmonic measure on convex domains}, Invent. Math {105}(1991)375-400.

\bibitem{Je96a} D. Jerison, {\it A Minkowski problem for electrostatic capacity}, Acta Math. {176}(1996)1-47.

\bibitem{Je96b} D. Jerison, {\it The direct method in the calculus of variations for convex bodies}, Adv. Math. {122}(1996)262-279.

\bibitem{JeK} D. Jerison and C. E. Kenig, {\it Boundary behavior of harmonic functions in nontangentially accessible domains}, Adv. Math. 46(1982)80-147.

\bibitem{JoNi61} F. John and L. Nirenberg, {\em On functions of bounded mean oscillation}, Comm. Pure Appl. Math. {14}(1961)415-426.

\bibitem{Kaw} B. Kawohl, {\it Overdetermined problems and the $p$-Lapalcian}, Proceedings of Equadiff 11(2005)1-6.

\bibitem{KicV} S. Kichenassamy and L. Veron, {\it Singular solutions of the $p$-Laplace equation}, Math. Ann. 275(1986)599-615.

\bibitem{Kol} S. Kolodziej, {\it The logarithmic capacity in $\mathbf C^n$}, Ann. Polon. Math. 48(1988)253-267.

\bibitem{Kub} T. Kubota, {\it \"Uber die konvex-geschlossenen Mannigfaltigkeiten im $n$-dimensionalen Raume}, Sci. Rep. Tojoku Univ. 14(1925)85-99.

\bibitem{Lam} M.-K. G. Lam, {\it The graph cases of the Riemannian postivie mass and Penrose inequalities in all dimensions}, arXiv:1010.4256v1[math.DG]20Oct2010.

\bibitem{Lan} N. S. Landkof, {\it Foundations of Modern Potential Theory}, Springer-Verlag, 1972.

\bibitem{Lew} J. Lewis, {\it Applications of Boundary Harnack Inequalities for $p$ Harmonic Functions and Related Topics}, C.I.M.E. Summer Course: Regularity Estimates for Nonlinear Elliptic and Parabolic Problems, Cetraro (Cesenza) Italy, June 21-27, 2009.

\bibitem{LewN} J. Lewis and K. Nystr\"om, {\it Boundary behaviour for $p$ harmonic functions in Lipschitz and starlike Lipschitz ring domains},  Ann. Sci. \'Ecole Norm. Sup. 40(2007)765-813.

\bibitem{LewN8} J. Lewis and K. Nystr\"om, {\it Regularity and free boundary regularity for the $p$-Laplacian in in Lipschitz and $C^1$-domains}, Ann. Acad. Sci. Fenn. Math. 33(2008)523-548.

\bibitem{Lie} G. M. Lieberman, {\it Boundary regularity for solutions of degenerate elliptic equations},  Nonlinear Anal. {12}(1988)1203-1219.

\bibitem{Lin} P. Lindqvist, {\it On the growth of the solutions of the differenital equation $div(|\nabla u|^{p-2}\nabla u)=0$ in $n$-dimensional space}, J. Diff. Equ. 58(1985)307-317.

\bibitem{Log} M. Longinetti, {\it Some isoperimetric inequalities for the level curves of capacity and Green's functions on convex plane domains}, SIAM J. Math. Anal. 19(1988)377-389.

\bibitem{MagPP} F. Maggi, M. Ponsiglione and A. Pratelli, {\it Quantitative stability in the isodiametric inequality via the isoperimetric inequality}, ArXiv:1104.4074v1 [math.MG] 20 Apr 2011.

\bibitem{MarRSY} O. Martio, V. Ryazanov, U. Srebro and E. Yakubov, {\it Moduli in Modern Mapping Theory}, Springer, 2009.

\bibitem{Maz} V. Maz\'ya, {\it Sobolev Spaces} with applications to elliptic partial differential equations, Springer-Verlag Berlin Heodelberg, 2011.

\bibitem{Mos} G. Mostow, {\it Quasi-conformal mappings in n-space and the rigidity of hyperbolic space forms}, Inst. Hautes \'Etudes Sci. Publ. Math. 34(1968)53-104.

\bibitem{MouY} L. Mou and P. Yang, Regularity for $n$-harmonic maps, {\it J. Geom. Anal.} 6(1996)91-112.

\bibitem{Nir} L. Nirenberg, {\it The Weyl and Minkowski problems in differential geometry in th elarge}, Comm. Pure Appl. Math. 6(1953)337-394.

\bibitem{PhiP} G. A. Philippin and L. E. Payne, {\it On the
conformal capacity problem}, Symposia Mathematica, Vol. XXX (Cortona) (1988)119-136.

\bibitem{Pog} A. V. Pogorelov, {\it The Minkowski Multidimensional Problem}, John Wiley \& Sons, 1978.

\bibitem{Pol} G. P\'olya, {\it Estimating electrostatic capacity}, Amer. Math. Monthly 54(1947)201-206.

\bibitem{PolS} G. P\'olya and G. Szeg\"o, {\it Aufgaben und Lehrs\"atze aus der Analysis II}, Berlin, G\"ottingen, Heidelberg, Springer-Verlag, 1954.

\bibitem{Pyr} P. Pyrih, {\it Loarithmic capacity is not subadditive - a fine topology approach}, Comment. Math. Univ. Carolin. 33,1(1992)67-72.

\bibitem{Ran} T. Ransford, {\it Potential Theory in the Complex Plane}, London Math. Soc. Student Texts {28}, Cambridge University Press 1995.


\bibitem{SafT} E. B. Saff and V. Totik, {\it Logarithmic Potentials with External Fields}, Springer-Verlag, 1997.

\bibitem{Sak} S. Sakaguchi, {\it Concavity properties of solutions to some degenerate quasilinear elliptic Dirichlet problems}, Ann. Scuola Norm. Sup. Pisa Cl. Sci. (4)14(1987), no. 3, 403-421(1988).

\bibitem{San} L. Santal\'o, {\it Integral Geometry and Geometric Probability}, Encyclopedia of Mathematics and Its Applications, Vol 1, Addison-Wesley, 1976.

\bibitem{Sch} M. Schiffer, {\it Hadamard's formula and variation of domain-functions}, Amer. J. Math. 68(1946)417–448.

\bibitem{Schn} R. Schneider, {\it Convex Bodies: The Brunn-Minkowski Theory}, Cambridge Univ. Press, 1993.

\bibitem{SY1} R. Schoen and S.-T. Yau, {\it On the proof of the positive mass conjecture in general relativity,} Comm. Math. Phys. 65(1979)45-76.

\bibitem{SY2} R. Schoen and S.-T. Yau, {\it On the structure of manifolds iwth positive scalar curvature,} Manuscripta Math. 28(1979)159-183.


\bibitem{Spr} J. Spruck, {\it Geometric aspects of the theory of fully non linear elliptic equations}, www.math.jhu.edu/~js/msri.notes.pdf.

\bibitem{Tol} P. Tolksdorf, {\it Regularity for a more general class of quasilinear elliptic equations}, J. Diff. Equa., 51(1984)126-150.

\bibitem{TrW} A. E. Treibergs and W. Wei, {\it Embedded hyperspheres with prescribed mean curvature}, J. Differential Geom. 18(1983)513–521.

\bibitem{Tso1} K. S. Tso, {\it Convex hypersufaces with prescribed Gauss-Kronecker curvature}, J. Differential Geom. 34(1991)389-410.

\bibitem{Tso2} K. S. Tso, {\it A direct method approach for the existence of convex hypersurfaces with prescribed Gauss-Kronecker curvature}, Math. Z. 209(1992)339-334.

\bibitem{UK} K. Uhlenbeck, {\it Regularity of a class of nonlinear elliptic systems}, Acta Math. 138(1970)219-240.

\bibitem{Wan} W. Wang, {\it N-capacity, N-harmonic radius and
N-harmonic transplantation}, J. Math. Anal. Appl. {327}(2007)155-174.

\bibitem{Wi} E. Witten, {\it A new proof of the positive energy theorem}, Comm. Math. Phys. 80(1981)381-402.

\bibitem{X} J. Xiao, {\it A Yau problem for variational capacity}, arXiv:1302.5132v4[mathDG]17Mar2013.

\bibitem{Yau} S.-T. Yau, {\it Problem section}. In: S.-T. Yau (ed.) Seminar on Differential Geometry, 669-706, Ann. Math. Stud. 102, Princeton University Press, Princeton, N.J., 1982.
\end{thebibliography}
\end{document}